\documentclass[10pt]{article}

\usepackage{amsmath,amsfonts,amssymb,amsthm,graphics,amscd}

\newtheorem{theorem}{Theorem}[section]
\newtheorem{lemma}{Lemma}[section]
\newtheorem{corollary}{Corollary}[section]

\newtheorem{example}[theorem]{Example}
\theoremstyle{definition}
\newtheorem{definition}{Definition}[section]

\newtheorem{proposition}[theorem]{Proposition}

\theoremstyle{remark}
\newtheorem{remark}[theorem]{Remark}

\def\ds{\displaystyle}

\def\ZZ{\mathbb{Z}}

\def\NN{\mathbb{N}}
\def\CC{\mathbb{C}}

\def\DD{\mathbb{D}}
\def\DDD{\bf{D}}
\def\inter{{\rm int}}
\numberwithin{equation}{section}

\def\acc{{\rm acc}}
\def\iso{{\rm iso}}

\def\ind{{\rm ind}}
\def\snoi{\smallskip\noindent}

\def\NN{{\mathbb N}}
\def\ZZ{{\mathbb Z}}
\def\CC{{\mathbb C}}

\def\DD{\mathbb{D}}
\def\W{\mathcal{W}}
\def\B{\mathcal{B}}
\def\M{\mathcal{M}}
\def\Q{\mathcal{Q}}
\def\D{\mathcal{D}}
\def\S{\mathcal{S}}

\begin{document}
\date{}
\title{Generalized Kato-Riesz decomposition}
\author{Sne\v zana \v C. \v Zivkovi\'c-Zlatanovi\'c, Milo\v s D. Cvetkovi\'c\footnote{The authors are
supported by the Ministry of Education, Science and Technological
Development, Republic of Serbia, grant no. 174007.}}

\date{}
\maketitle

\setcounter{page}{1}

\begin{abstract}
We shall say that a bounded linear operator $T$ acting on a Banach space $X$ admits a generalized  Kato-Riesz
decomposition if there exists a pair of $T$-invariant closed subspaces $(M,N)$ such that $X=M\oplus N$, the reduction  $T_M$ is Kato and $T_N$ is Riesz. In this paper we define and investigate the generalized Kato-Riesz spectrum of an operator. For $T$ is said to be    generalized Drazin-Riesz
invertible if there exists a bounded linear operator $S $ acting on $X$ such that
$TS=ST$, $STS=S$, $ TST-T$ is Riesz.
 We  investigate generalized Drazin-Riesz
invertible operators and
also,  characterize bounded linear operators  which can be expressed as a direct sum of a Riesz operator and a bounded below (resp.  surjective, upper (lower) semi-Fredholm, Fredholm,  upper (lower)
semi-Weyl, Weyl) operator. In particular we characterize the single-valued extension property at  a point $\lambda_0\in{\mathbb C}$  in the case that $\lambda_0-T$ admits  a generalized  Kato-Riesz decomposition.
\end{abstract}



\section{Introduction and preliminaries}


Let $\mathbb{N} \, (\mathbb{N}_0)$ denote the set of all positive
(non-negative) integers, and let $\mathbb{C}$ denote the set of all
complex numbers. If $K \subset \mathbb{C}$, then $\partial K$ is the
boundary of $K$, $\acc \, K$ is the set of accumulation points of
$K$, $\iso \, K=K \setminus \acc \, K$ and $\inter\, K$ is the set of interior points of $K$.
For $\lambda_0\in \CC$, the open disc, centered  at $\lambda_0$  with radius $\epsilon$ in $\CC$, is denoted  by $D(\lambda_0,\epsilon)$.

 Let $L(X)$ be the Banach algebra of bounded linear operators acting on  an infinite dimensional complex  Banach space $X$.  The group of all
invertible operators is denoted by $L(X)^{-1}$.
  For $T\in L(X)$, let $\sigma(T)$ denote the spectrum of $T$, $r(T)$ the spectral radius of $T$  and  $\rho(T)=\CC\setminus\sigma(T)$  the resolvent set of $T$. Let   $N(T)$ denote the null-space and $R(T)$ the range of $T$. Set $ \alpha (T) = \text{dim}
N(T)$  and  $\beta (T)=\text{codim} R(T)$. An operator $T \in L(X)$ is {\em bounded
below} if $T$ is injective and has closed range. Let $\M(X)$ denote the set of all
bounded below operators, and let $\Q(X)$ denote the set of all
surjective operators.
   The approximate point spectrum of $T\in L(X)$
is defined by
$
\sigma_{ap}(T)=\{\lambda\in\CC: T-\lambda\ {\rm is\ not\ bounded\
below}\}
$
and the surjective spectrum is defined by
$
\sigma_{su}(T)=\{\lambda\in\CC: T-\lambda\ {\rm is\ not\
surjective}\}.
$ We set $\rho_{ap}(T)=\CC\setminus \sigma_{ap}(T)$
and $\rho_{su}(T)=\CC\setminus \sigma_{su}(T)$.
  An operator $T \in
L(X)$ is {\em Kato} if $R(T)$ is closed and $N(T) \subset
R(T^n)$ for every $ n \in \mathbb{N}$. An operator  $T\in L(X)$ is  {\em nilpotent} when $T^n=0$ for some
$n \in \mathbb{N}$, while $T$ is {\em quasinilpotent} if
$\|T^n\|^{1/n}\to 0$, that is $T-\lambda \in L(X)^{-1}$  for all
complex $\lambda\ne 0$.

 Sets of {\it upper and
lower semi-Fredholm} operators, respectively, are defined as
 $\Phi_+(X)=\{ T \in L (X)\, :\, \alpha (T)<\infty\text{
 and
}R(T)\text{ is closed}\}, $
 and
 $
           \Phi_-(X)= \{ T  \in L (X) \, : \, \beta (T)<\infty
           \}.
           $
             For upper and
lower semi-Fredholm operators  the index is defined by
           $\ind (T)=\alpha (T)-\beta (T)$. If
           $T\in\Phi_+(X)\backslash\Phi_-(X)$, then $\ind (T)=-\infty$,
           and if $T\in\Phi_-(X)\backslash\Phi_+(X)$, then
           $\ind(T)=+\infty$.
                      The set of Fredholm operators is defined as
           $
\Phi (X)=  \Phi_+(X) \cap \Phi_-(X). $
 The sets  of upper semi-Weyl, lower semi-Weyl and  Weyl operators are
defined as $\W_+(X)=\{ T\in\Phi_+ (X): \ind (T)\le 0\}$, $\W_-(X)=\{ T\in\Phi_- (X): \ind (T)\ge 0\}$  and  $\W(X)=\{ T\in\Phi (X): \ind (T)=0\}$, respectively.
Denote by
$\text{asc}(A)$ ($\text{dsc}(A)$) the ascent (the descent) of
$A\in B(X)$, i.e.  the smallest  $n\in\NN_0$ such that
$N(A^n)=N(A^{n+1})$ ($R(A^n)=R(A^{n+1})$). If such $n$ does not
exist, then $\text{asc} (A)=\infty$ ($\text{dsc} (A)=\infty$). An
operator $A\in B(X)$ is {\it Browder} if it is Fredholm and has
finite ascent and finite descent.
 An operator $A\in B(X)$ is called {\it upper
semi-Browder} if it is upper semi-Fredholm of finite ascent, and
{\it lower semi-Browder} if it is lower semi-Fredholm of finite
descent.  Let $\mathcal{B}(X)$  (resp. $\mathcal{B}_+(X)$, $\mathcal{B}_-(X)$)
denote the set of all Browder  (resp. upper semi-Browder, lower semi-Browder)  operators.
Corresponding  to these classes  we have the following spectra and resolvent sets: the upper Fredholm spectrum $\sigma_{\Phi_+}(T)$ of $T$ is defined by $\sigma_{\Phi_+}(T)=\{\lambda\in\CC: T-\lambda\notin\Phi_+(X)\}$ and the upper Fredholm resolvent set is $\rho_{\Phi_+}(T)=\CC\setminus \sigma_{\Phi_+}(T)$, and similarly for the  lower Fredholm (resp. Fredholm, upper (lower) Weyl, Weyl, upper (lower) Browder, Browder) spectrum and resolvent set.

An operator $T\in L(X)$ is Riesz, $T\in R(X)$,  if
$T-\lambda\in\Phi (X)$ for all $\lambda\in \CC\backslash\{ 0\}$.
We shall say that an operator $T\in B(X)$ is {\it polynomially
Riesz} and write $T\in{\rm Poly}^{-1}R(X)$ if there exists a
nonzero complex polynomial $p(z)$ such that $p(T)\in R(X)$. Recall that if $T\in{\rm Poly}^{-1}R(X)$, then
there exists a unique  polynomial $\pi_T$ of
minimal degree with leading coefficient 1  such that $\pi_T(T)\in R(X)$ which we call the minimal polynomial of $T$
(see \cite{ZDHB}).

If $M$ is a subspace of $X$ such that $T(M) \subset M$, $T \in
L(X)$, it is said that $M$ is {\em $T$-invariant}. We define $T_M:M
\to M$ as $T_Mx=Tx, \, x \in M$.  If $M$ and $N$ are two closed
$T$-invariant subspaces of $X$ such that $X=M \oplus N$, we say that
$T$ is {\em completely reduced} by the pair $(M,N)$ and it is
denoted by $(M,N) \in Red(T)$. In this case we write $T=T_M \oplus
T_N$ and say that $T$ is a {\em direct sum} of $T_M$ and $T_N$.

An operator $T \in L(X)$ is said to admit a {\em generalized Kato
decomposition}, abbreviated as GKD, if there exists a pair $(M,N)
\in Red(T)$ such that $T_M$ is Kato and $T_N$ is quasinilpotent. A
relevant case is obtained if we assume that $T_N$ is nilpotent. In
this case $T$ is said to be of {\em Kato type}. An operator is said
to be {\em essentially Kato} if it admits a GKD $(M,N)$ such that
$N$ is finite-dimensional. If $T$ is essentially Kato then $T_N$ is
nilpotent, since every quasinilpotent operator on a finite
dimensional space is nilpotent. 
The semi-Fredholm operators
 belong to the class of essentially Kato operators
\cite[Theorem 16.21]{Mu}.
For $T\in L(X)$, {\em the Kato spectrum},
{\em the  essentially Kato spectrum}, { \em the Kato type spectrum}
and {\it the generalized Kato spectrum} is defined by

$ \sigma_K(T) = \{\lambda \in \CC: T-\lambda \ \text{ is\ not\ Kato}\}$,

$\sigma_{eK}(T) = \{\lambda \in \CC: T-\lambda\ \text{ is\ not\
essentially\ Kato}\}$,

$ \sigma_{Kt}(T)  =  \{\lambda \in \CC: T-\lambda\ \text{ is\ not\
of\ Kato\ type}\}$,

$ \sigma_{gK}(T) = \{\lambda \in \CC: T-\lambda\ \text{ does\ not\
admit\ a\  GKD}\}$,  respectively.

\medskip
\begin{definition}
 An operator $T \in L(X)$ is said to  admit a {\em  Kato-Riesz
decomposition}, abbreviated as GKRD, if there exists a pair $(M,N)
\in Red(T)$ such that $T_M$ is Kato and $T_N$ is Riesz.
\end{definition}

For $T\in L(X)$ {\it the generalized Kato-Riesz spectrum} is defined by
$$
\sigma_{gKR}(T)=\{\lambda\in\CC: T-\lambda\ {\rm does\ not\ admit\  a\ GKRD}\}.
$$

Clearly, $\sigma_{gKR}(T)\subset\sigma_{gK}(T)\subset \sigma_{Kt}(T)\subset
\sigma_{eK}(T)\subset\sigma_{K}(T)\subset \sigma_{ap}(T)\cap
\sigma_{su}(T)$.

An operator $T\in L(X)$ is said to be Drazin invertible if there exists $S\in L(X)$ such that $TS=ST$, $STS=S$  and $TST-T$ is nilpotent. This concept has been  generalized by Koliha \cite{koliha} by replacing the third condition in this definition 
 with the condition that $TST-T$ is quasinilpotent. Recall that $T$ is generalized Drazin invertible if and only if $0\notin\acc\, \sigma (T)$, and this is also equivalent to the fact that
 $T=T_1\oplus T_2$ where $T_1$ is invertible and $T_2$ is quasinilpotent. If we replace the third condition in the previous definitions by condition that $TST-T$ is Riesz, we get the concept of generalized Drazin-Riesz invertible operators.
  The generalized Drazin-Riesz  spectrum of $T\in L(X)$ is
defined by
$
\sigma_{gDR}(T)=\{\lambda\in\CC: T-\lambda\ {\rm is\ not\
generalized\ Drazin-Riesz\ invertible}\}.
$
In the second section of this paper we prove that an operator $T\in L(X)$ is generalized Drazin-Riesz invertible if and only if $T$ admits a GKRD and $0$ is not an interior point of $\sigma (T)$ and this is also equivalent to the fact that $T=T_1\oplus T_2$ where $T_1$ is invertible and $T_2$ is Riesz.
Also
 we characterize operators which are a direct sum of a Riesz operator  and a bounded below  (resp. surjective,  upper (lower)  semi-Fredholm, upper (lower) semi-Weyl, upper (lower)semi-Browder) operator. These operators generalize the class of generalized Drazin invertible operators and also the class of generalized Drazin-Riesz invertible  operators,  and hence we shall call them {\it generalized Drazin-Riesz bounded below (resp. generalized Drazin-Riesz surjective,  generalized Drazin-Riesz upper (lower) semi-Fredholm, generalized Drazin-Riesz Fredholm, etc) 
 operators}, and we shall
  use the following notations:
\begin{center}
\begin{tabular}{|c|c|c|} \hline
${\bf R_1}=\M(X)$ & ${\bf R_2}=\Q(X)$ & ${\bf R_3}=L(X)^{-1}$ \\
\hline
${\bf R_4}=\mathcal{B}_+(X)$ & ${\bf R_5}=\mathcal{B}_-(X)$ & ${\bf R_6}=\mathcal{B}(X)$ \\
\hline
${\bf R_7}=\Phi_+(X)$ & ${\bf R_8}=\Phi_-(X)$ & ${\bf R_9}=\Phi(X)$ \\
\hline
${\bf R_{10}}=\mathcal{W}_+(X)$ & ${\bf R_{11}}=\mathcal{W}_-(X)$ & ${\bf R_{12}}=\mathcal{W}(X)$ \\
\hline
\end{tabular}
\end{center}
\noindent  

$${\bf g DR R}_i(X)=\{ T \in L(X) : T=T_1\oplus T_2,\ T_1\in {\bf R}_i,\ T_2\ {\rm is\ Riesz}\},\ 1\le i\le 12.
$$

%

An operator $T\in L(X)$ is said to have the single-valued extension property at $\lambda_0\in\CC$ (SVEP at $\lambda_0$ for breviety) if for every open disc $\D_{\lambda_0}$ centerd at $\lambda_0$ the only analitic function $f:\D_{\lambda_0}\to X$ satisfying $(T-\lambda)f(\lambda)=0 $ for all $\lambda\in \D_{\lambda_0}$ is the function $f\equiv 0$.
An operator $T\in L(X)$ is said to have the SVEP if $T$ has the SVEP at every point $\lambda\in\CC$.
We denote by $\S(T)$ the open set of $\lambda\in\CC$ where $T$ fails to have  SVEP at $\lambda$.

Evidently, $T\in L(X)$ has the SVEP at every point of the resolvent set $\rho (T)$. Moreover, from the identity theorem for analytic function and  $\sigma(T)=\sigma(T^\prime)$, where  $T^\prime\in L(X^\prime )$ is the adjoint operator of $T$, it follows that $T$ and $T^\prime$ have the SVEP at every point of the boundary   $\partial\sigma(T)$ of the spectrum. In particular, $T$ and $T^\prime$ have the SVEP at every isolated  point of the spectrum. Hence, there is implication
\begin{equation}\label{w1}
 \sigma(T)\ {\rm does\ not\ cluster\ at\ }\lambda_0\Longrightarrow T\ {\rm and\ }T^\prime\ {\rm have\ the\ SVEP\ at\ }\lambda_0.
\end{equation}
Moreover, from the indentity theorem for analytic functions  we have (see \cite{aienarosas}, p. 182):
\begin{equation}\label{w2}
  \sigma_{ap}(T)\ {\rm does\ not\ cluster\ at\ }\lambda_0\Longrightarrow T\ {\rm has\ the\ SVEP\ at\ }\lambda_0
\end{equation}
and
\begin{equation}\label{w3}
  \sigma_{su}(T)\ {\rm does\ not\ cluster\ at\ }\lambda_0\Longrightarrow T^\prime\ {\rm has\ the\ SVEP\ at\ }\lambda_0.
\end{equation}
P. Aiena and E. Rosas proved that if $\lambda_0-T$ is of Kato type, then the implications \eqref{w1}, \eqref{w2} and \eqref{w3} can be reversed \cite{aienarosas}. Q. Jiang and H. Zhong \cite{kinezi}  showed that if $\lambda_0-T$ admits a GKD,  then the following statements are equivalent:

(i) $T$\ ($T^\prime$) has the SVEP at $\lambda_0$;

(ii) $\sigma_{ap}(T)$ ($\sigma_{su}(T)$) does not cluster at $\lambda_0$;

(iii) $\lambda_0$ is not an interior point of $\sigma_{ap}(T)$ ($\sigma_{su}(T)$),

\noindent that is, the implications \eqref{w1}, \eqref{w2} and \eqref{w3}   can be also reversed in the case that $\lambda_0-T$ admits a GKD.
We extend  this result to the case of operators which admit a GKRD. Precisely, we show that  if $\lambda_0-T$ admits a GKRD,  then the following statements are equivalent:

(i) $T\ (T^\prime)\ {\rm has\ the\ SVEP\ at\ }\lambda_0$;

(ii) $\lambda_0$ is not an interior point of $\sigma_{ap}(T)\ (\sigma_{su}(T))$;

(iii) $\sigma_{\B_+}(T)\ (\sigma_{\B_-}(T))$ does not cluster at $\lambda_0$;

(iv) $\lambda_0$ is not an interior point of $\sigma_{\B_+}(T)\ (\sigma_{\B_-}(T))$.

A Riesz operator $T$ with the  infinite spectrum  is an example of an operator which admits GKRD and has the SVEP, but the spectra $\sigma_{ap}(T)=\sigma_{su}(T)=\sigma(T)$ cluster at $0$. So, if $\lambda_0-T$ admits a GKRD, then the statement that $T\ (T^\prime)\ {\rm has\ the\ SVEP\ at\ }\lambda_0$ is not in general equivalent to the statement that  $\sigma_{ap}(T)$ ($\sigma_{su}(T)$) does not cluster at $\lambda_0$.

Also, we extend the previous results to the cases of essential spectra. Namely, we prove that if
$\lambda_0-T$ admits a GKRD,  then  $\lambda_0$ is not an interior point of $\sigma_{{\bf R}}(T)$
if and only if
 $\sigma_{{\bf R}}(T)$ does not cluster at $\lambda_0$ where ${\bf R}$ is one of $\Phi_+, \Phi_-, \Phi, \W_+, \W_-, \W$. In that way we extend results obtained in \cite{ZC}.

The third section is devoted to investigation of corresponding spectra.  We prove that the generalized Kato-Riesz spectrum is compact and that it differs from the essential Kato spectrum on at most countably  many points. As one application we get for  $T$ which is a unilateral weighted right shift operator on $\ell_p(\NN)$, $1\le p<\infty$, with weight $(\omega_n)$, and  $c(T)=\lim\limits_{n\to\infty}\inf (\omega_1\cdot\dots\cdot \omega_n)^{1/n}=0$, that  $\sigma_{gKR}(T)= \sigma(T)=\overline{D(0,r(T))}$.  Also we
obtain that  the connected hulls of the
generalized Drazin-Riesz spectrum and  the generalized Kato-Riesz spectrum ot $T\in L(X)$ are
equal and hence they are empty in the same time. Moreover, we prove that it happens if and only if  $T$ is a polynomially Riesz operator.  We prove  that
$\partial\sigma_{\bf R}(T)\cap\acc\, \sigma_{\bf R}(T)\subset \sigma_{gKR}(T)$ where $\bf R$ is one of $\Phi_+, \Phi_-, \Phi, \W_+, \W_-, \W$, $\B_+$, $\B_-$, $\B$, and also
$  \partial\sigma_{ap}(T)\cap {\acc}\, \sigma_{\B_+}(T)\subset\sigma_{gKR}(T)$,
   $\partial\sigma_{su}(T)\cap {\acc}\, \sigma_{\B_-}(T)\subset\sigma_{gKR}(T)$ and $
    \partial\sigma(T)\cap {\acc}\, \sigma_{\B}(T)\subset\sigma_{gKR}(T)$.
As consequence we get that if
 $\sigma_{ap}(T)=\partial\sigma (T)$ ($\sigma_{su}(T)=\partial\sigma (T)$) and every $\lambda\in \partial\sigma (T)$ is not isolated in $\sigma (T)$,  then
$\sigma_{gKR}(T)=\sigma_{ap}(T)$ ($\sigma_{gKR}(T)=\sigma_{su}(T)$). These results are applied to some concrete cases, as  the forward and backward bilateral shifts on $c_0(\ZZ),\ell_p(\ZZ)$,  the forward and backward unilateral shifts on $  c_0(\NN), c(\NN), \ell_{\infty}(\NN)$ or $\ell_p(\NN)$, $p\ge 1$, as well as  arbitrary non-invertible isometry, and also $\rm Ces\acute{a}ro$ operator.

\section{Generalized Drazin-Riesz invertible and generalized  Drazin-Riesz semi-Fredholm  operators}

We start with the following several assertions which will be useful in the sequel.
\begin{lemma} (\cite[Lemma 2.2]{ZC})\label{lema}
Let $T \in L(X)$ and $(M,N) \in Red(T)$. The following statements
hold:
\par \noindent {\rm (i)} $T \in {\bf R}_i$ if and only if $T_M
\in {\bf R}_i$ and $T_N \in {\bf R}_i$, $1 \leq i \leq 9$, and in
that case $\ind (T)=\ind(T_M)+\ind(T_N)$;\par \noindent {\rm (ii)}
If $T_M \in {\bf R}_i$ and $T_N \in {\bf R}_i$, then $T \in {\bf
R}_i$, $10 \leq i \leq 12$.

\snoi{\rm (iii)} If $T \in {\bf R}_i$ and $T_N$ is Weyl, then
$T_M \in {\bf R}_i$, $10 \leq i \leq 12$.
\end{lemma}
%
%

\begin{lemma}\label{lema-Riesz}{(\cite[Lemma 2.11]{ZDHB})}
Let $T\in L(X)$ and let $(M,N)\in Red (T)$. Then
  $T$ is
Riesz if and only if $T_M$ and $T_N$ are Riesz.
\end{lemma}
The following proposition  is fundamental for our purpose.
\begin{proposition} \label{gap0}(\cite[Proposition 3.1]{ZC})
Let $T \in L(X)$. Then the following implications hold:
\par \noindent {\rm (i)} If $T$ is Kato and $0$ is an accumulation point of $\rho_{ap}(T)$ (resp. $ \rho_{su}(T)$,  $\rho(T)$), then $T$ is bounded below (resp. surjective, invertible);

\par \noindent {\rm (ii)} If $T$ is Kato and $0$ is an accumulation point of $\rho_{\B_+}(T)$ (resp. $\rho_{\B_-}(T)$, $\rho_{\B}(T)$), then $T$ is bounded below (resp. surjective, invertible);

\par \noindent {\rm (iii)} If $T$ is Kato and $0$ is an accumulation point of $\rho_{\Phi_+}(T)$ (resp. $\rho_{\Phi_-}(T)$, $\rho_{\Phi}(T)$), then $T$ is upper  semi-Fredholm (resp. lower semi-Fredholm, Fredholm);

\par \noindent {\rm (iv)} If $T$ is Kato and $0$ is an accumulation point of $\rho_{\W_+}(T)$ (resp. $\rho_{\W_-}(T)$, $\rho_{\W}(T)$), then $T$ is upper semi-Weyl (resp. lower semi-Weyl, Weyl).
\end{proposition}

For a subspace $M$ of $X$ its annihilator $M^\bot$ is defined by
$$
M^\bot=\{f\in X^\prime:f(x)=0\ {\rm for\ all\ } x\in M\}.
$$

\begin{proposition} \label{decomp}
Let $T\in L(X)$. If $T$ admits a GKRD$(M,N)$, then
$(N^\bot, M^\bot)$ is a GKRD for $T^\prime$.
\end{proposition}
\begin{proof} Suppose that $T$ admits a GKRD$(M,N)$. Then  $(M,N)\in Red (T)$, $T_M$ is Kato, $T_N$ is Riesz and $(N^\bot, M^\bot)\in Red (T^\prime)$.
Let $P_M$ be the projection  of $X$ onto $M$ along $N$. Then $P_M^\prime$ is also a projector, $N(P_M^\prime)=R(P_M)^\bot=M^\bot$ and  $R(P_M^\prime)=N(P_M)^\bot=N^\bot$ since $R(P_M)$ is closed. Consequently,  $X^\prime=R(P_M^\prime)\oplus N(P_M^\prime)= N^\bot\oplus M^\bot$. For $P_N=I-P_M$ we have that $(M,N)\in Red (TP_N)$,  $TP_N=P_NT$ and since $TP_N=(TP_N)_M\oplus (TP_N)_N=0\oplus T_N$,  according to Lemma \ref{lema-Riesz}, it follows that  $TP_N$ is Riesz.
\ Consequently, $T^\prime P_N^\prime=P_N^\prime T^\prime$ is Riesz, $(N^\bot, M^\bot)\in Red (T^\prime P_N^\prime )$ and since  $R(P_N^\prime)=N(P_N)^\bot=M^\bot$,  we conclude that $(T^\prime P_N^\prime)_{M^\bot}={T^\prime}_{M^\bot}$ is Riesz according to Lemma \ref{lema-Riesz}. From the proof of Theorem 1.43 in \cite{Ai} it follows  that ${T^\prime}_{N^\bot}$ is Kato. Therefore, $(N^\bot, M^\bot)$ is a GKRD for $T^\prime$.
\end{proof}

\begin{definition}
An operator $T\in L(X)$ is  {\em generalized Drazin-Riesz
invertible} if there exists $S \in L(X)$ such that
\[TS=ST, \; \; \; STS=S, \; \; \; TST-T \; \; \text{is Riesz}.\]
\end{definition}
Clearly, every Riesz operator is generalized Drazin-Riesz
invertible.
The set of generalized Drazin-Riesz invertible operators is denoted by
$L(X)^{DR}$.
  Harte introduced the concept of a quasipolar element in a Banach algebra \cite[Definition 7.5.2]{H}:  an element $a$ of a Banach algebra $A$ is {\it quasipolar} if there is an idempotent $q\in A$ commuting with $a$ such that $a(1-q)$ is quasinilpotent and $q\in (Aa)\cap(aA)$. Koliha \cite{koliha} proved that an element is quasipolar if and only if it is generalized Drazin invertible.
We shall say that $T\in L(X)$  is {\it Riesz-quasipolar} if there exists a bounded projection $Q$ satisfying
\begin{equation}\label{Riesz-qp}
  TQ=QT,\ T(I-Q)\ {\rm is\ Riesz},\ Q\in(L(X)T)\cap (TL(X)).
\end{equation}
In the following theorem we show that an operator $T\in L(X)$  is  Riesz-quasipolar if and only if $T$ is generalized Drazin-Riesz invertible and moreover, it is also equivalent to the fact that $T$ admits a GKRD and $T$ and $T^\prime$ have the SVEP at $0$.

\begin{theorem} \label{t3}
Let $T\in L(X)$. The following
conditions are equivalent:

\par \noindent {\rm
(i)} There exists $(M,N)\in Red(T)$ such that $ T_M$ is invertible   and $ T_N$ is Riesz;

\par\noindent {\rm (ii)} $T$ admits a GKRD and
$0\notin{\rm int}\, \sigma(T)$.

\par\noindent {\rm (iii)} $T$ admits a GKRD and $T$ and $T^\prime$ have the SVEP at $0$;

\par \noindent {\rm (iv)} $T$ is generalized Drazin-Riesz invertible;

\par \noindent {\rm (v)} $T$ is Riesz-quasipolar;

\par \noindent {\rm (vi)}
There exists a bounded projection $P$ on $X$ which commutes with $T$
such that $T+P$ is Browder and $TP$ is Riesz;

\par \noindent {\rm
(vii)} There exists $(M,N)\in Red(T)$ such that $ T_M$ is Browder   and $ T_N$ is Riesz;

\par \noindent {\rm (viii)} $T$ admits a GKRD  and
$0\notin {\rm acc}\, \sigma_{\mathcal{B}}(T)$;

\par \noindent {\rm (ix)} $T$ admits a GKRD  and
$0\notin {\rm int}\, \sigma_{\mathcal{B}}(T)$.
\end{theorem}
\begin{proof}

{\rm (i)} $\Longrightarrow$ {\rm (ii)}: Suppose that
  there exists $(M,N)\in Red(T)$ such that $ T_M$ is invertible  and $ T_N$ is Riesz. Then $T_M$ is Kato and hence,  $T$ admits a
GKRD $(M,N)$. Since $T_M$ is invertible, $0\in\rho_{ap}(T_M)$ and there exists $\epsilon>0$ such that $D(0,\epsilon)\subset\rho(T_M)$. As $T_N$ is Riesz, it follows that  $0\in\acc\, \rho(T_N)$. Consequently, $0\in \acc\, (\rho(T_M)\cap\rho(T_N))=\acc\, \rho(T)$  and so, $0\notin{\rm int}\, \sigma(T)$.

{\rm (ii)} $\Longrightarrow$ {\rm (i)}: Suppose that $T$ admits a
GKRD and $0\notin{\rm int}\, \sigma(T)$. Then there exists $(M,N) \in
Red(T)$ such that $T_M$ is Kato and $T_N$ is Riesz and $0  \in \acc \, \rho(T)$. According to Lema
\ref{lema}{\rm (i)}, it follows that $0  \in \acc \,
\rho(T_M)$. From Proposition \ref{gap0} {\rm (i)}  it follows
that $T_M$
 is invertible.

{\rm (ii)} $\Longrightarrow$ {\rm (iii)}: Follows from the fact that $\sigma(T)=\sigma(T^\prime)$ and the identity theorem for analytic functions.

{\rm (iii)} $\Longrightarrow$ {\rm (ii)}: Suppose that $T$ admits a GKRD and $T$ and $T^\prime$ have the SVEP at $0$. Then there exists $(M,N)\in Red(T)$ such that $ T_M$ is Kato and $ T_N$ is Riesz. Since the SVEP at 0 of $T$ is inherited by the reductions on every closed invariant subspaces, we get that $T_M$ has the SVEP at $0$. According to \cite[Theorem 2.49]{Ai} it follows that $T_M$ is bounded below and so there exists $\epsilon_1>0$ for which  $T_M-\lambda$ is bounded below for every $|\lambda|<\epsilon_1$, that is $D(0,\epsilon_1)\subset\rho_{ap}(T_M)$. Since $T_N$ is Riesz, then $D(0,\epsilon_1)\setminus C_1\subset\rho_{ap}(T_N)$ where $C_1$ is most countable set of Riesz points of $T_N$. Consequently, $D(0,\epsilon_1)\setminus C_1\subset\rho_{ap}(T_M)\cap\rho_{ap}(T_N)=\rho_{ap}(T)$.  From Proposition \ref{decomp} it follows that $T^\prime$ admits the  GKRD $(N^\bot,M^\bot)$ and since $T^\prime$ has the SVEP at $0$, according to already proved we conclude that there exists $\epsilon_2>0$ such that $D(0,\epsilon_2)\setminus C_2
\subset\rho_{ap}(T^\prime)=\rho_{su}(T)$ where $C_2$ is most countable set of Riesz points of $T^\prime_{M^\bot}$. Let $\epsilon=\min\{\epsilon_1,\epsilon_2\}$ and $C=C_1\cup C_2$. Than $C$ is most countable and $D(0,\epsilon)\setminus C\subset \rho_{ap}(T)\cap\rho_{su}(T)=\rho(T)$. Consequently, $0\notin\inter \, \sigma(T)$.

(i)$\Longrightarrow$(iv): Suppose that there exists $(M,N)\in Red(T)$ such that $ T_M$ is invertible   and $ T_N$ is Riesz. Let $S=T_M^{-1}\oplus 0$, i.e.
$$
S=\bmatrix T_M^{-1}&0\\0&0
\endbmatrix
:\bmatrix M\\N
\endbmatrix
\to \bmatrix M\\N
\endbmatrix
$$
and  let $x\in X$. Then $x=m+n$, where $m \in M$ and $n \in N$, and
 $TSx=TS(m+n)=TT_M^{-1}m=m$ and $STx=ST(m+n)=S(T_Mm+T_Nn)=T_M^{-1}T_Mm=m$. Thus $TS=ST$. Further,
$STSx=STS(m+n)=Sm=S(m+n)=Sx$ and hence, $STS=S$. From
\begin{eqnarray*}
  T-T^2S &=& \bmatrix T_M&0\\0&T_N\endbmatrix-\bmatrix T_M^2&0\\0&T_N^2\endbmatrix\cdot \bmatrix T_M^{-1}&0\\0&0
\endbmatrix \\
   &=& \bmatrix T_M&0\\0&T_N\endbmatrix-\bmatrix T_M&0\\0&0\endbmatrix\\&=&\bmatrix 0&0\\0&T_N\endbmatrix,
\end{eqnarray*}
according to Lemma \ref{lema-Riesz}, it follows that $ T-T^2S$ is Riesz.

(iv)$\Longrightarrow$(v):  Suppose that $T$ is generalized Drazin-Riesz invertible. Then there exists $S\in L(X)$ such that $ST=TS$, $STS=S$ and $T-T^2S$ is Riesz. Let $Q=TS$. Then $Q$ is a bounded projection which commutes with $T$, $Q=TS=ST\in (L(X)T)\cap (TL(X))$ and  $T(I-Q)=T(I-TS)=T-T^2S$ is Riesz.

(v)$\Longrightarrow$(vi): Suppose  that $T$ is Riesz-quasipolar. Then there exists a bounded projection $Q$ satisfying (\ref{Riesz-qp}).
Let $P=I-Q$. Then $TP=PT$ and $TP$ is Riesz. From $I-P=Q\in (L(X)T)\cap (TL(X))$ it follows that there exist $U,V\in L(X)$ such that $I-P=UT=TV$. Then
\begin{equation}\label{d1}
  (T+P)(UTV+P)=(UTV+P)(T+P)=I+TP.
\end{equation}
Since $TP$ is Riesz, by \cite[Theorem 3.111]{Ai} it follows that $I+TP$ is Browder. Therefore from (\ref{d1}), according to \cite[Theorem 7.9.2]{H}, we obtain that $T+P$ is Browder.

{\rm (vi)}
$\Longrightarrow$ {\rm (vii)} : Suppose that
there exists a projection $P \in L(X)$ that commutes with $T$ such
that $T+P $ is Browder  and $TP$ is Riesz. For $M=N(P)$ and $N=R(P)$ we have that $(M,N)\in Red(T)$ and $T_N=(TP)_N$ is Riesz by Lemma \ref{lema-Riesz}.
\ From Lemma
\ref{lema}{\rm (i)}  it follows that  $T_M=(T+P)_M $ is Browder.

\smallskip
 {\rm (vii)} $\Longrightarrow$ {\rm (viii)}:  Let $(M,N)\in Red(T)$  and $T=T_M \oplus T_N$, where $T_M $ is Browder  and $T_N$ is Riesz. Then $0\in\rho_{{\bf \B}}(T_M)$ and there exists $\epsilon>0$ such that $D(0,\epsilon)\subset\rho_{{\bf \B}}(T_M)$. Since $T_N$ is Riesz, $\sigma_{{\bf \B}}(T_N) \subset\{0\}$ by \cite[Theorem 3.111]{Ai} and hence,  $D(0,\epsilon)\setminus\{0\}\subset\rho_{{\bf \B}}(T_M)\cap\rho_{{\bf \B}}(T_N)$. According to Lemma \ref{lema}{\rm (i)} $\rho_{{\bf \B}}(T_M)\cap\rho_{{\bf \B}}(T_N)=\rho_{{\bf \B}}(T)$ and so, $D(0,\epsilon)\setminus\{0\}\subset\rho_{{\bf \B}}(T)$. Therefore, $0\notin \acc\, \sigma_{{\bf \B}}(T)$.

 From   \cite[Theorem
16.21]{Mu} it follows that there exist two closed $T$-invariant
subspaces $M_1$ and $M_2$ such that $M=M_1 \oplus M_2, \, M_2 \,
\text{is finite dimensional}$, $T_{M_1}$ is Kato and
$T_{M_2}$ is nilpotent. Hence $X=M_1 \oplus (M_2 \oplus N)$ and $M_2 \oplus N$ is closed. From Lemma \ref{lema-Riesz} it follows
that $T_{M_2 \oplus N}=T_{M_2} \oplus T_N$ is Riesz and
thus $T$ admits the GKRD $(M_1,M_2 \oplus N)$.

\smallskip
 {\rm (viii)} $\Longrightarrow$ {\rm (ix)}: Obvious.

\smallskip
 {\rm (ix)} $\Longrightarrow$ {\rm (i)}:
 Suppose that $T$ admits a
GKRD and $0\notin{\rm int}\, \sigma_{\mathcal{B}}(T)$. Then  there exists $(M,N) \in
Red(T)$ such that $T_M$ is Kato and $T_N$ is Riesz. Since $0  \in \acc \, \rho_{\mathcal{B}}(T)$,
from Lema
\ref{lema}{\rm (i)} we get that  $0  \in \acc \,
\rho_{\mathcal{B}}(T_M)$, which according to  Proposition \ref{gap0} {\rm (ii)} implies
that $T_M$
 is invertible.
\end{proof}


\medskip

Evidently, every Riesz operator  and every Browder operator is generalized Drazin-Riesz invertible operator.

Using Proposition \ref{gap0} ((i), (ii)), similarly to the proof of  Theorem \ref{t3}, the following theorems can be proved.

\begin{theorem} \label{t1}
Let $T\in L(X)$. The following
conditions are equivalent:

\par \noindent {\rm
(i)} There exists $(M,N)\in Red(T)$ such that $ T_M$ is bounded
below and $ T_N$ is Riesz, that is
 $T\in {\bf g DR{\bf\mathcal{M}}}(X)$;


\par\noindent {\rm (ii)} $T$ admits a GKRD and
$0\notin{\rm int}\, \sigma_{ap}(T)$;

\par\noindent {\rm (iii)} $T$ admits a GKRD and $T$ has the SVEP at $0$;

\par \noindent {\rm
(iv)} There exists $(M,N)\in Red(T)$ such that $ T_M$ is upper semi-Browder and $ T_N$ is Riesz, that is $T\in {\bf g DR{\mathcal{B}_+}}(X)$;

\par \noindent {\rm (v)} $T$ admits a GKRD  and
$0\notin {\rm acc}\, \sigma_{\mathcal{B}_+}(T)$;

\par \noindent {\rm (vi)} $T$ admits a GKRD  and
$0\notin {\rm int}\, \sigma_{\mathcal{B}_+}(T)$;

\par \noindent {\rm (vii)}
There exists a bounded projection $P$ on $X$ which commutes with $T$
such that $T+P$ is uper semi-Browder and $TP$ is Riesz.
\end{theorem}

\begin{theorem} \label{t2}
Let $T\in L(X)$. The following
conditions are equivalent:

\par \noindent {\rm
(i)} There exists $(M,N)\in Red(T)$ such that $ T_M$ is surjective and $ T_N$ is Riesz, that is
 $T\in {\bf g DR{\bf\mathcal{Q}}}(X)$;


\par\noindent {\rm (ii)} $T$ admits a GKRD and
$0\notin{\rm int}\, \sigma_{su}(T)$;

\par\noindent {\rm (iii)} $T$ admits a GKRD and $T^\prime$ has the SVEP at $0$;

\par \noindent {\rm
(iv)} There exists $(M,N)\in Red(T)$ such that $ T_M$ is lower semi-Browder and $ T_N$ is Riesz, that is $T\in {\bf g DR{\mathcal{B}_-}}(X)$;

\par \noindent {\rm (v)} $T$ admits a GKRD  and
$0\notin {\rm acc}\, \sigma_{\mathcal{B}_-}(T)$;

\par \noindent {\rm (vi)} $T$ admits a GKRD  and
$0\notin {\rm int}\, \sigma_{\mathcal{B}_-}(T)$;

\par \noindent {\rm (vii)}
There exists a bounded projection $P$ on $X$ which commutes with $T$
such that $T+P$ is lower semi-Browder and $TP$ is Riesz.
\end{theorem}

\medskip
The condition that $0\notin{\rm int}\, \sigma_{ap }(T)$    ($0\notin{\rm int}\, \sigma_{su}(T)$) in the statement (ii) in Theorem \ref{t1} (Theorem \ref{t2}) can not be replaced with the stronger condition that $0\notin{\rm acc}\, \sigma_{ap}(T)$ ($0\notin{\rm acc}\, \sigma_{su}(T)$). The example which shows that is a Riesz operator with infinite spectrum. Namely, if $T\in L(X)$ is Riesz with infinite spectrum, then obviously $T$ is generalized Drazin-Riesz invertible, but $0\in{\rm acc}\, \sigma_{ap}(T)={\rm acc}\, \sigma_{su}(T)$ and $0\notin{\rm int}\, \sigma_{ap }(T)={\rm int}\, \sigma_{su }(T)$.


\begin{corollary} Let $T\in L(X)$. If $\lambda_0-T$ admits a GKRD,  then the following statements are equivalent:

(i) $T\ (T^\prime)\ {\rm has\ the\ SVEP\ at\ }\lambda_0$;

(ii) $\lambda_0$ is not an interior point of $\sigma_{ap}(T)\ (\sigma_{su}(T))$;

(iii) $\sigma_{\B_+}(T)\ (\sigma_{\B_-}(T))$ does not cluster at $\lambda_0$;

(iv) $\lambda_0$ is not an interior point of $\sigma_{\B_+}(T)\ (\sigma_{\B_-}(T))$.
\end{corollary}
\begin{proof} Follows from the equivalences (ii)$\Longleftrightarrow$(iii)$\Longleftrightarrow$(v)$\Longleftrightarrow$(vi) in Theorem \ref{t1} (Theorem \ref{t2}).
\end{proof}

\begin{theorem} \label{t4}
Let $T\in L(X)$ and $7\le i\le 12$. The following
conditions are equivalent:

\par \noindent {\rm
(i)} There exists $(M,N)\in Red(T)$ such that $ T_M\in {\bf R}_i$ and $T_N$ is Riesz, that is
 $T\in {\bf g DR R_i}(X)$;

\par \noindent {\rm (ii)} $T$ admits a GKRD  and
$0\notin {\rm acc}\, \sigma_{{\bf R}_i}(T)$;

\par\noindent {\rm (iii)} $T$ admits a GKRD and
$0\notin{\rm int}\, \sigma
_{{\bf R}_i}(T)$;

\par \noindent {\rm (iv)}
There exists a bounded projection $P$ on $X$ which commutes with $T$
such that $T+P\in {\bf R}_i$  and $TP$ is Riesz.
\end{theorem}
\begin{proof} (i)$\Longrightarrow$(ii) Suppose that there exists $(M,N)\in Red(T)$ such that $ T_M\in {\bf R}_i$ and $T_N$ is Riesz.
As in the proof of the implication (vii)$\Longrightarrow$(viii) in Theorem \ref{t3} it follows that $T$ admits a GKRD.

   Since ${\bf R}_i$ is open, from $ T_M\in {\bf R}_i$ it follows that there exists $\epsilon>0$ such that $D(0,\epsilon)\subset\rho_{{\bf R}_i}(T_M)$. Since $T_N$ is Riesz,  according to \cite[Theorem 3.111]{Ai}, it follows that  $\sigma_{{\bf R}_i}(T_N) \subset\{0\}$ and hence,  $D(0,\epsilon)\setminus\{0\}\subset\rho_{{\bf R}_i}(T_M)\cap\rho_{{\bf R}_i}(T_N)$. According to Lemma \ref{lema}{\rm (i)},{\rm (ii)}, $\rho_{{\bf R}_i}(T_M)\cap\rho_{{\bf R}_i}(T_N)\subset\rho_{{\bf R}_i}(T)$ and so, $D(0,\epsilon)\setminus\{0\}\subset\rho_{{\bf R}_i}(T)$. Therefore, $0\notin \acc\, \sigma_{{\bf R}_i}(T)$.

(ii)$\Longrightarrow$(iii) Obvious.

(iii)$\Longrightarrow$(i)
 Suppose that $T$ admits a
GKRD and $0\notin{\rm int}\, \sigma_{\Phi_+}(T)$. Then there exists $(M,N) \in
Red(T)$ such that $T_M$ is Kato and $T_N$ is Riesz and $0  \in \acc \, \rho_{\Phi_+}(T)$. According to Lema
\ref{lema}{\rm (i)}, it follows that $0  \in \acc \,
\rho_{\Phi_+}(T_M)$. From Proposition \ref{gap0} {\rm (iii)}  it follows
that $T_M$
 is upper semi-Fredholm.
The cases $i=8$ and $i=9$ can be proved similarly.

Suppose that $T$ admits a GKRD and $0  \in \acc \, \rho_{\W_+}(T)$.
Then  there exists $(M,N) \in Red(T)$ such that $T_M$ is Kato and
$T_N$ is Riesz. We show that $0  \in \acc \,
\rho_{\W_+}(T_M)$. Let $\epsilon>0$. From $0  \in \acc \,
\rho_{\W_+}(T)$ it follows that there exists $\lambda\in\CC$ such
that $0<|\lambda|<\epsilon$  and $T-\lambda\in\W_+(X)$. As $T_N$ is
Riesz, $T_N- \lambda$ is Fredholm of index zero, and so, according to
Lema \ref{lema}{\rm (iii)}, we conclude that
$T_M-\lambda\in\W_+(M)$, that is $\lambda\in \rho_{\W_+}(T_M)$.
Therefore, $0  \in \acc \, \rho_{\W_+}(T_M)$ and  from Proposition
\ref{gap0} {\rm (iv)}  it follows that $T_M$
 is upper semi-Weyl,  and so $T \in {\bf g DR \W_+}(X)$. The cases $i=11$ and $i=12$ can be proved similarly.

{\rm (i)} $\Longrightarrow$ {\rm (iv)}:  Suppose that there exists $(M,N) \in Red(T)$ such that
$T_M\in {\bf R}_i$ and $T_N$ is Riesz. Let $P \in L(X)$
be a projection such that $N(P)=M$ and $R(P)=N$.  Then $TP=PT$ and since  $TP=(TP)_M\oplus(TP)_N=0\oplus T_N$ and $T_N$ is Riesz, from Lemma \ref{lema-Riesz} it follows that $TP$ is Riesz.
Also from the fact that
 $T_N$ is Riesz it follows that $\sigma_{{\bf R}_i}(T_N)\subset\{0\}$ and so, $(T+P)_N=T_N+I_N \in
{\bf R}_i$, where $I_N$ is identity on $N$. Since  $(T+P)_M=T_M\in {\bf R}_i $, we have  that    $T+P\in {\bf R}_i $
by Lemma \ref{lema}{\rm (i)}, {\rm (ii)}.

\smallskip

 {\rm (iv)} $\Longrightarrow$ {\rm (i)}: Suppose that
there exists a projection $P \in L(X)$ that commutes with $T$ such
that $T+P\in {\bf R}_i $   and $TP$ is Riesz. For $M=N(P)$ and $N=R(P)$ we have that $(M,N)\in Red(T)$ and $T_N=(TP)_N$ is Riesz. For $i\in\{7,8,9\}$ from Lemma
\ref{lema}{\rm (i)}  it follows that  $T_M=(T+P)_M\in {\bf R}_i $. Suppose that $i\in\{10,11,12\}$. Since $T_N$ is Riesz, it follows that $T_N+I_N$ is Weyl
\ and so, from  $T+P=(T+P)_M\oplus (T+P)_N=T_M\oplus (T_N+I_N)$, according to Lemma \ref{lema}{\rm (iii)}, it follows that $T_M\in  {\bf R}_i $.
\end{proof}

The following corollary is an improvement of Corollary 3.4 in \cite{ZC}.
\begin{corollary} \label{t4-GKt0} Let $T\in L(X)$ and $7\le i\le 12$. If $\lambda_0-T$ admits a GKRD,  then the following statements are equivalent:

(i) $\lambda_0$ is not an interior point of $\sigma_{{\bf R}_i}(T)$;

(ii) $\sigma_{{\bf R}_i}(T)$ does not cluster at $\lambda_0$.
\end{corollary}
\begin{proof}  Follows from the equivalence {\rm (ii)}$\Longleftrightarrow${\rm (iii)} in Theorems \ref{t4}.
\end{proof}
\begin{corollary} \label{t4-GKt} Let $T\in L(X)$ and let $0\in\partial\sigma_{{\bf R}_i}(T), \, 1 \leq i \leq 12$. Then $T$ admits a
generalized Kato-Riesz decomposition if and only if $T$ belongs to  ${\bf gDRR}_i(X)$.
\end{corollary}
\begin{proof}  Follows from the equivalences {\rm (i)}$\Longleftrightarrow${\rm (ii)} in Theorems \ref{t3}, \ref{t1}, \ref{t4} and the equivalence
{\rm (i)}$\Longleftrightarrow${\rm (iii)} in Theorem \ref{t4}.
\end{proof}

\medskip

Let $T\in L(X)$ be generalized Drazin-Riesz invertible  operator $T$, that is, there exists $(M,N)\in Red(T)$ such that $ T_M$ is invertible   and $ T_N$ is Riesz. For the operator $T_N$ we shall say that it is a Riesz part of $T$.

It is well known that a  Riesz operator with infinite spectrum has the property that the sequence of its Riesz points converges to $0$. In the following proposition we show that a generalized Drazin-Riesz invertible operator $T$,  which Riesz part $T_N$ has infinite spectrum, has the same property: there exists a sequence of nonzero Riesz points of $T$ which converges to $0$. Moreover, it holds the converse: if $T$ admits a GKRD and there exists a sequence of nonzero Riesz points of $T$ which converges to $0$, then $T$ is generalized Drazin-Riesz operator  which Riesz part  has infinite spectrum.

\begin{proposition}\label{Riesz}
Let $T \in L(X)$. The following statements are equivalent:\par

\medskip

\noindent {\rm (i)}  $T=T_M \oplus
T_N$ where  $T_M$ is invertible and  $T_N$ is Riesz with infinite
spectrum;\par

\smallskip

\noindent {\rm (ii)} $T$ admits a GKRD and there exists a sequence of nonzero Riesz points of $T$ which converges to $0$.
\end{proposition}
\begin{proof}
{\rm (i)} $\Longrightarrow$ {\rm (ii)}: Suppose that $T=T_M \oplus
T_N$ where  $T_M$ is invertible and  $T_N$ is Riesz with infinite
spectrum. Then $T$ admits a GKRD$(M,N)$ and
$\sigma(T_N)=\{0, \mu_1, \mu_2, \ldots\}$ where $\mu_n$, $n\in\NN$,  are nonzero Riesz points of $T_N$   and
\begin{equation}\label{MM}
  \lim_{n \to \infty}
\mu_n=0.
\end{equation}
 According to Theorem \ref{t3} we have  that  $0 \not \in \acc
\, \sigma_{\bf \B}(T)$, i.e. there exists $\epsilon>0$ such that $\mu \not
\in \sigma_{\bf \B}(T)$ for $0<|\mu|<\epsilon$. From (\ref{MM}) it follows that there exists $n_0 \in
\NN$ such that $0<|\mu_n|<\epsilon$ for $n\ge n_0$. Hence $\mu_n \in \sigma(T) \setminus
\sigma_{\bf \B}(T)$ for all $n\ge n_0$ and, since  the set $\sigma(T) \setminus
\sigma_{\bf \B}(T)$ is exactly the set of the Riesz points of $T$, we see that $(\mu_n)_{n=n_0}^{\infty}$ is the sequence of nonzero Riesz points of $T$ which converges to $0$.
\par

\smallskip

{\rm (ii)} $\Longrightarrow$ {\rm (i)}: Suppose that $T=T_M \oplus
T_N$ where $T_M$ is Kato, $T_N$ is Riesz  and let $(\lambda_n)$ is the sequence of nonzero  Riesz points of $T$ such
that $0=\lim_{n \to \infty} \lambda_n$. Since  $\lambda_n
\in \rho_{\bf \B}(T)$ for all $n \in \NN$, it follows that $0 \in \acc \,
\rho_{\bf \B}(T)$.  As in the proof of Theorem \ref{t3} we conclude that $T_M$ is invertible. Thus there exists an $\epsilon>0$ such that $D(0,\epsilon)\subset \rho(T_M)$ and there exists $n_0\in\NN$ such that $\lambda_n\in  D(0,\epsilon)$ for all $n\ge n_0$. Consequently, $\lambda_n\notin \sigma(T_M)$ for all $n\ge n_0$ and since $\lambda_n\in \sigma(T)=\sigma(T_M)\cup \sigma(T_N)$, it follows that $\lambda_n\in \sigma(T_N)$ for all $n\ge n_0$. Therefore, the spectrum of $T_N$ is infinite.
\end{proof}
\begin{corollary}
Let $T \in L(X)$ be generalized Drazin-Riesz invertible and  let $0\in{\rm acc}\, \sigma(T)$. Then
there exists a sequence of nonzero Riesz points
of $T$ which converges to $0$.
\end{corollary}
\begin{proof}
 According to
Theorem \ref{t3} it follows that $T=T_M \oplus T_N$ with $T_M$
invertible and $T_N$ Riesz.  Since $0\in{\rm acc}\, \sigma(T)$, it follows that $0\in{\rm acc}\, \sigma_N(T)$ and so, $\sigma_N(T)$ is infinite.  Applying Proposition \ref{Riesz}  we
obtain that there exists a sequence of nonzero Riesz points of $T$ which converges to $0$.
\end{proof}

\begin{theorem} \label{calculus1+}
Let $T \in L(X)$ and let $f$ be a
complex analytic function in a neighborhood of $\sigma(T)$.
If $T\in {\bf g DR R}_i(X)$ and $f^{-1}(0) \cap \sigma_{\bf{R}_i}(T)=\{0\}$, then $f(T)\in {\bf g DR R}_i(X)$, $1\le i\le 12$.
\end{theorem}
\begin{proof} We shall prove the assertion for the case $i=4$ and $i=10$. Suppose that $T\in {\bf g DR R}_4(X)$. Then, according to Theorem \ref{t3}, $T$ is  generalized Drazin-Riesz bounded below and there exists $(M,N) \in Red(T)$
such that $T_M$ is bounded below and $T_N$ is Riesz. The pair $(M,N)$
completely reduces $(\lambda-T)^{-1}$ for every $\lambda \in
\rho(T)$ and so,
$f(T)=\frac{1}{2\pi i}  \int\limits_{\gamma} f(\lambda)(\lambda-T)^{-1}d\lambda$, where $\gamma$ is a contour
surrounding $\sigma(T)$ and which lies in the domain of $f$,  is also reduced by the pair $(M,N)$, 
  $f(T)_M=f(T_M)$ and $f(T)_N=f(T_N)$, and consequently
$f(T)=f(T_M) \oplus f(T_N)$.

Suppose also that
$f^{-1}(0) \cap \sigma_{\mathcal{B}_+}(T)=\{0\}$.
Using the fact that $0 \not \in \sigma_{\mathcal{B}_+}(T_M) \subset
\sigma_{\mathcal{B}_+}(T)$, we obtain
$0 \not \in f(\sigma_{\mathcal{B}_+}(T_M))$. According to the spectral
mapping theorem it follows $0 \not \in
f(\sigma_{\mathcal{B}_+}(T_M))=\sigma_{\mathcal{B}_+}(f(T_M))$ \cite[Theorem 3.4]{Voki}, so
$f(T_M)$ is upper semi-Browder. Since $f(0)=0$, it follows that  $f(T_N)$ is Riesz by
\cite[Theorem 3.113 (i)]{Ai}.  Consequently,  $f(T)\in {\bf g DR R}_4(X)$, i.e. $f(T)$ is generalized Drazin-Riesz bounded below.

The  cases for $i=1,2,3,5,6,7,8,9$ can be proved similarly.

Suppose that $T$ is   generalized Drazin-Riesz upper semi-Weyl and
$f^{-1}(0) \cap \sigma_{{\W}_+}(T)=\{0\}$. Then there exists $(M,N) \in Red(T)$
such that $T_M$ is upper semi-Weyl and $T_N$ is Riesz. As above we conclude that $f(T)=f(T_M) \oplus f(T_N)$ and $f(T_N)$ is Riesz. From
$0 \not \in \sigma_{\W_+}(T_M) \subset
\sigma_{\W_+}(T)$, we obtain
$0 \not \in f(\sigma_{\W_+}(T_M))$. Since $\sigma_{\W_+}(f(T_M))\subset f(\sigma_{\W_+}(T_M))$ \cite[Theorem 3.3]{Voki}, it follows that
$0 \not \in
\sigma_{\W_+}(f(T_M))$, and  so
$f(T_M)$ is upper semi-Weyl.  Consequently,  $f(T)$ is generalized Drazin-Riesz upper semi-Weyl.

 Similarly for  $i=11,12$.
\end{proof}

\noindent The following corollary follows at once from Theorem

\ref{calculus1+}.

\begin{corollary}
Let $T \in L(X)$. If $T \in {\bf g DR R}_i(X)$, then $T^n \in {\bf g D
RR}_i(X)$ for every $n \in \mathbb{N}$,  $1 \leq i \leq 12$.
\end{corollary}

\begin{proposition} \label{decomposition}
Let $T \in L(X)$ and let $f$ be a complex analytic function in a
neighborhood of $\sigma(T)$ such that $f^{-1}(0) \cap \acc \,
\sigma(T)=\emptyset$. Then $f(T)=A+K$, where $A \in L(X)$ is
generalized Dazin-Riesz Fredholm and $K \in L(X)$ is compact.
\end{proposition}
\begin{proof}
Since $\sigma(\pi(T)) \subset \sigma(T)$, $f$ is analytic in a
neighborhood of $\sigma(\pi(T))$ and we have $f(\pi(T))=\pi(f(T))$.
 Then
\[\acc \, \sigma(\pi(f(T))=\acc \, \sigma(f(\pi(T))
\subset f(\acc \, \sigma(\pi(T)) \subset f(\acc \, \sigma(T))\] according to
\cite[Theorem 2]{Harte}. By the assumption it follows $0 \not \in
\acc \, \sigma(\pi(f(T))$, i.e. $\pi(f(T))$ is generalized Drazin
invertible. Now, we apply \cite[Theorem 3.11]{Boasso1}, which
completes the proof.
\end{proof}

\begin{corollary}
Let $T \in L(X)$ have finite spectrum and let $f$ be a
complex analytic function in a neighborhood of $\sigma(T)$. Then
$f(T)=A+K$, where $A \in L(X)$ is generalized Dazin-Riesz Fredholm
and $K \in L(X)$ is compact.
\end{corollary}
\begin{proof}
Since
$\acc \, \sigma(T)=\emptyset$, so the condition $f^{-1}(0) \cap \acc
\, \sigma(T)=\emptyset$ is automatically satisfied, and apply
Proposition \ref{decomposition}.
\end{proof}

\begin{corollary}
Let $T \in L(X)$ be polynomially Riesz and let $f$ be a
complex analytic function in a neighborhood of $\sigma(T)$ such that $f^{-1}(0) \cap
\pi_T^{-1}(0)=\emptyset$  where $\pi_T$ is the minimal polynomial of $T$. Then
$f(T)=A+K$, where $A \in L(X)$ is generalized Dazin-Riesz Fredholm
and $K \in L(X)$ is compact.
\end{corollary}
\begin{proof}
Notice that if $T \in L(X)$ is polynomially Riesz, then
$\acc \, \sigma(T)\subset \sigma_{\mathcal{B}}(T)=\pi_T^{-1}(0)$, so  $f^{-1}(0) \cap \acc
\, \sigma(T)\subset f^{-1}(0) \cap
\pi_T^{-1}(0)=\emptyset$  and we apply
Proposition \ref{decomposition}.
\end{proof}

\section{Spectra}
For $T\in L(X)$ we define the spectra with respect to the sets ${\bf
g DR R}_i(X)$:

\[\sigma_{{\bf g DR R}_i}(T)=\{\lambda \in \mathbb{C}:T-\lambda \not \in
{\bf g DR R}_i (X)\}, \; \; 1\leq i \leq 12.\]

Instead of $\sigma_{\bf gDRL(X)^{-1}}(T)$ we shall write simpler  $\sigma_{ gDR}(T)$.
From Theorems \ref{t1},  \ref{t2} and \ref{t3} it follows that
\begin{eqnarray}
\sigma_{\bf gDR\M}(T)&=&\sigma_{gKR}(T)\cup {\rm int} \,
\sigma_{ap} (T)\label{n0}\\
&=&\sigma_{gKR}(T)\cup {\rm int} \,
\sigma_{\B_+} (T)\nonumber\\
&=&\sigma_{gKR}(T)\cup \acc\, \sigma_{\B_+} (T)),\nonumber
\\&=&\sigma_{gKR}(T)\cup \S (T),\nonumber
\end{eqnarray}
\begin{eqnarray}
\sigma_{\bf gDR\Q}(T)&=&\sigma_{gKR}(T)\cup {\rm int}\, \sigma_{su} (T)\label{n1}\\
&=&\sigma_{gKR}(T)\cup {\rm int}\, \sigma_{\B_-} (T)\nonumber\\
&=&\sigma_{gKR}(T)\cup \acc\, \sigma_{\B_-} (T),\nonumber
\\&=&\sigma_{gKR}(T)\cup \S (T^\prime)\nonumber
\end{eqnarray}
and
\begin{eqnarray}
\sigma_{gDR}(T)&=&\sigma_{gKR}(T)\cup {\rm int }\,  \sigma(T)\label{n2}\\
&=&\sigma_{gKR}(T)\cup {\rm int}\, \sigma_{\B} (T)\nonumber\\
&=&\sigma_{gKR}(T)\cup \acc\, \sigma_{\B} (T)\label{poziv}
\\&=&\sigma_{gKR}(T)\cup\S(T)\cup \S (T^\prime).\nonumber
\end{eqnarray}

From Theorem \ref{t4} it follows that
\begin{eqnarray}
 \sigma_{{\bf g DR R}_i}(T)&=&\sigma_{gKR}(T) \cup \acc \, \sigma_{{\bf
R}_i}(T)\label{glava}\\&=&\sigma_{gKR}(T) \cup {\rm int} \,
\sigma_{{\bf R}_i}(T), \ 7\leq i \leq 12.\label{glava-}
\end{eqnarray}

 Clearly,

%

\begin{eqnarray}  &\sigma_{\bf g DR \Phi_+}(T)\subset \sigma_{\bf g DR \W_+}(T)\subset \sigma_{\bf g D R\M}(T)\subset \sigma_{ap}(T)\nonumber\\
  \sigma_{gKR}(T)\subset &\label{inkluzija}\\
  &\sigma_{\bf g DR \Phi_-}(T)\subset \sigma_{\bf g DR \W_-}(T)\subset \sigma_{\bf g DR \Q}(T)\subset \sigma_{su}(T),\nonumber
\end{eqnarray}
 \begin{eqnarray}  &\sigma_{\bf g DR \Phi_+}(T)\subset \sigma_{\bf g DR \W_+}(T)\subset \sigma_{\bf g DR \M}(T)&\nonumber\\
  \sigma_{gKR}(T)\subset & &\subset \sigma_{gDR}(T)\label{inkluzija-}\\
  &\sigma_{\bf g DR \Phi_-}(T)\subset \sigma_{\bf g DR \W_-}(T)\subset \sigma_{\bf g DR \Q}(T)&\nonumber
\end{eqnarray}
and

 \begin{equation}\label{inkluzija--}
   \sigma_{gKR}(T)\subset\sigma_{\bf g D R\Phi}(T)\subset \sigma_{\bf g DR \W}(T)\subset\sigma_{gDR}(T).
 \end{equation}
%

\medskip

%
%

\begin{remark}\label{poi} {\rm We remark that
\begin{eqnarray}
   &\Phi_+(X)\setminus\W_+(X)\subset {\bf gDR\Phi_+}(X)\setminus {\bf gDR\W_+}(X),\label{gh0}\\
   &\Phi_-(X)\setminus\W_-(X)\subset {\bf gDR\Phi_-}(X)\setminus {\bf gDR\W_-}(X),\label{gh}\\
   &\Phi(X)\setminus\W(X)\subset {\bf gDR\Phi}(X)\setminus {\bf gDR\W}(X).\label{gh00}
\end{eqnarray}
Indeed, as the index is locally constant, the set
$\Phi_+(X)\setminus\W_+(X)=\{T\in \Phi(X): \ind(T)>0\}$ is open,
which implies that the set
$\sigma_{\W_+}(T)\setminus\sigma_{\Phi_+}(T)=\rho_{\Phi_+}(T)\setminus\rho_{\W_+}(T)$
is open for every $T\in L(X)$. Suppose that $T\in
\Phi_+(X)\setminus\W_+(X)$. Then $T\in {\bf gDR \Phi_+}(X)$ and $0\in
\sigma_{\W_+}(T)\setminus\sigma_{\Phi_+}(T)$. There exists
$\epsilon>0$ such that $D(0,\epsilon)\subset
\sigma_{\W_+}(T)\setminus\sigma_{\Phi_+}(T)$. Hence, $0\in \inter\,
\sigma_{\W_+}(T)$ and $T\notin {\bf gDR\W_+}(X)$  according to
Theorem \ref{t4}. Similarly for (\ref{gh}) and (\ref{gh00}).}
\end{remark}
\noindent The following example shows that the inclusions $\sigma_{{\bf gDR\Phi_+}}(T)\subset \sigma_{{\bf gDR\W_+}}(T)$, $\sigma_{{\bf gDR\Phi_-}}(T)\subset \sigma_{{\bf gDR\W_-}}(T)$ and $\sigma_{{\bf gDR\Phi}}(T)\subset \sigma_{{\bf gDR\W}}(T)$ can be proper.
\begin{example} \label{ex2}
{\rm If  $X$ is one of  $c_0(\NN), c(\NN), \ell_{\infty}(\NN), \ell_p(\NN)$, $p\ge 1$,\ 
the forward and the backward unilateral shifts $U$ and $V$ on $X$ are
Fredholm, $\ind(U)=-1$ and $\ind(V)=1$. Therefore,   $U\in
\Phi_-(X)\setminus\W_-(X)$ and $V\in  \Phi_+(X)\setminus\W_+(X)$, and also $U,V\in
\Phi(X)\setminus\W(X)$.
Hence, according to Remark \ref{poi},
 $U\in  {\bf gDR\Phi_-}(X)\setminus {\bf gDR\W_-}(X)$, $V\in {\bf gDR\Phi_+}(X)\setminus {\bf gDR\W_+}(X)$ and $U,V\in {\bf gDR\Phi}(X)\setminus {\bf gDR\W}(X)$.
 This implies that $0\in \sigma_{{\bf gDR\W_-}}(U)\setminus\sigma_{{\bf gDR\Phi_-}}(U)$, $0\in \sigma_{{\bf gDR\W_+}}(V)\setminus\sigma_{{\bf gDR\Phi_+}}(V)$ and $0\in \sigma_{{\bf gDR\W}}(U)\setminus \sigma_{{\bf gDR\Phi}}(U)$.}
\end{example}
\noindent The following example shows that the inclusions $\sigma_{{\bf gDR\W_+}}(T)\subset\sigma_{{\bf
gDR\M}}(T)$ and $\sigma_{{\bf gDR\W_-}}(T)\subset\sigma_{{\bf
gDR\Q}}(T)$ can be proper.
Set $\DD=\{\lambda\in\CC:|\lambda|\le 1\}$.

\begin{example} \label{ex3}
{\rm If  $X$ is one of  $c_0(\NN), c(\NN), \ell_{\infty}(\NN), \ell_p(\NN)$, $p\ge 1$,  let $T=U\oplus V$ where  $U$ and $V$ are forward and the backward unilateral shifts, respectively. Then, according to Lemma \ref{lema}(i),  $T$ is Fredholm and $\ind(T)=\ind(U)+\ind(V)=0$. Thus $T$ is Weyl and hence, $T$ is generalized Drazin-Riesz Weyl. Since $\sigma_{ap}(U)= \sigma_{su}(V)=\partial\DD$ and $\sigma_{su}(U)= \sigma_{ap}(V)=\DD$, it follows that $\sigma_{ap}(T)=\sigma_{ap}(U)\cup \sigma_{ap}(V)=\DD$ and $\sigma_{su}(T)=\sigma_{su}(U)\cup \sigma_{su}(V)=\DD$. Therefore, $0\in\inter\, \sigma_{ap}(T)$ and $0\in\inter\, \sigma_{su}(T)$ and  from Theorems \ref{t1} and \ref{t2} it follows  that $T$ is neither generalized Drazin-Riesz bounded below nor generalized Drazin-Riesz surjective and so, $0\in \sigma_{{\bf
gDR\M}}(T)\setminus \sigma_{{\bf gDR\W_+}}(T)$ and $0\in \sigma_{{\bf
gDR\Q}}(T)\setminus \sigma_{{\bf gDR\W_-}}(T)$.
}
\end{example}

 We need the following result.
\begin{proposition}{(\cite[p. 143]{Mull-Mb})}\label{Muller-Mbehta} Let $T\in L(X)$ and $(M,N)\in Red(T)$. Then $T$ is essentially  Kato if and only if $T_M$ and $T_N$ are essentially Kato.
\end{proposition}
\begin{theorem}\label{closed}
Let $T\in L(X)$ and let $T$ admits a GKRD$(M,N)$. Then there exists $\epsilon>0$ such that $T-\lambda$ is essentially Kato for each $\lambda$ such that  $0<|\lambda|<\epsilon$.
\end{theorem}
\begin{proof}
If $M=\{0\}$, then $T$ is Riesz and hence $T-\lambda$ is Fredholm for all $\lambda\ne 0$. From \cite[Theorem 16.21]{Mu} it follows that $T-\lambda$ is essentially Kato for all $\lambda\ne 0$.

Suppose that $M\ne \{0\}$. From \cite[Theorem 1.31]{Ai} it follows that $|\lambda|<\gamma(T_M)$, $T_M-\lambda$ is Kato. Since $T_N$ is Riesz, then $T_N-\lambda$ is esentially Kato for all $\lambda\ne 0$. Let $\epsilon=\gamma(T_M)$. From Proposition \ref{Muller-Mbehta} it follows that $T-\lambda$ is essentially Kato for each $\lambda$ such that  $0<|\lambda|<\epsilon$.
\end{proof}
\begin{corollary}
Let $T\in L(X)$. Then $\sigma_{gKR}(T)$ is compact  and the sets $\sigma_{eK}\setminus\sigma_{gKR}(T)$,  $\sigma_{Kt}(T)\setminus\sigma_{gKR}(T)$ and $\sigma_{gK}(T)\setminus\sigma_{gKR}(T)$ consist of  at most countably many points.
\end{corollary}
\begin{proof} From Theorem \ref{closed} it follows that $\sigma_{gKR}(T)$ is closed and since $\sigma_{gKR}(T)\subset \sigma(T)$ it follows that $\sigma_{gKR}(T)$ is bounded. Thus, $\sigma_{gKR}(T)$ is compact.

Let $\lambda_0\in \sigma_{eK}(T)\setminus\sigma_{gKR}(T)$. Then $T-\lambda_0$ admits a GKRD and according to Theorem \ref{closed} there exists $\epsilon>0$ such that $T-\lambda$ is essentially Kato for each $\lambda$ such that  $0<|\lambda-\lambda_0|<\epsilon$. This means that $\lambda_0\in\iso\, \sigma_{eK}(T)$. Therefore $\sigma_{eK}(T)\setminus\sigma_{gKR}(T)\subset \iso\, \sigma_{eK}(T)$ and since $\sigma_{gK}(T)\subset \sigma_{Kt}(T)\subset \sigma_{eK}(T)$, it follows that $\sigma_{Kt}(T)\setminus\sigma_{gKR}(T)\subset \iso\, \sigma_{eK}(T)$ and $\sigma_{gK}(T)\setminus\sigma_{gKR}(T)\subset \iso\, \sigma_{eK}(T)$, which implies that $\sigma_{eK}\setminus\sigma_{gKR}(T)$,  $\sigma_{Kt}(T)\setminus\sigma_{gKR}(T)$ and $\sigma_{gK}(T)\setminus\sigma_{gKR}(T)$ are  at most countable.
\end{proof}

\begin{proposition} \label{cor1}
For  $T \in L(X)$ the following statements
hold:
\par \noindent {\rm (i)} $\sigma_{{\bf g DR R}_i}(T) \subset
\sigma_{{\bf R}_i}(T) \subset \sigma(T)$, $1\leq i \leq 12$;\par \noindent {\rm (ii)}
$\sigma_{{\bf g D RR}_i}(T)$ is a compact subset of $\mathbb{C}$, $1\leq i \leq 12$;\par
\noindent {\rm (iii)} $\sigma_{{\bf R}_i}(T) \setminus \sigma_{{\bf
g D RR}_i}(T)$ consists of  at most countably many isolated points, $4\le i\le 12$.
\end{proposition}
\begin{proof}
{\rm (i):} It is obvious.\par \noindent {\rm (ii):} It suffices to
prove that the complement of $\sigma_{{\bf g DR R}_i}(T)$ is open. If $\lambda_0 \not \in \sigma_{{\bf g D
R}\M}(T)$, then $T-\lambda_0 \in {\bf g D R}\M(X)$ and by Theorem
\ref{t1} there exists $\epsilon>0$ such that $T-\lambda_0-\lambda
\in {\bf \B_+}(X)\subset {\bf g D R}\B_+(X)={\bf {g D R}\M}(X) $ for $0<|\lambda|<\epsilon$. Thus, $T-\lambda_0-\lambda
\in {\bf {g D R}\M}(X) $ for each $\lambda$ such that $|\lambda|<\epsilon$, that is
 $D(\lambda_0,\epsilon)\subset\CC\setminus\sigma_{{\bf g D R}\M}(T)$. Therefore, $\CC\setminus\sigma_{{\bf g D R}\M}(T)$ is open.
For the other cases we apply similar consideration.
\par \noindent {\rm (iii):} If $\lambda\in \sigma_{\B_+}(T) \setminus \sigma_{{\bf
g D R}\B_+}(T)$, then $\lambda \in \sigma_{\B_+}(T)$ and
$T-\lambda\in {\bf g D R}\B_+(X)$. From Theorem \ref{t1} we
obtain that  $\lambda \in  \iso \, \sigma_{\B_+}(T)$ and thus,
$\sigma_{\B_+}(T) \setminus \sigma_{{\bf g D R}\B_+}(T)$ consists of
at most countably many isolated points. Similarly for the other cases when $5\le i\le 12$.
\end{proof}

\begin{corollary}\label{r=}
Let $T\in L(X)$  and $i\in\{1,\dots,12\}$. If $T$  is an operator
for which
 \begin{equation}\label{i4}
  \sigma_{{\bf R}_i}(T)=\partial\sigma_{{\bf R}_i}(T),
\end{equation}
 then
\begin{equation}\label{i14}
  \sigma_{gKR}(T)=\sigma_{{\bf gDRR}_i}(T).
\end{equation}
In particular, if
\begin{equation}\label{i}
  \sigma(T)=\partial\sigma (T),
\end{equation}
 then
\begin{equation}\label{i1}
  \sigma_{gKR}(T)=\sigma_{gDR}(T).
\end{equation}
\end{corollary}
\begin{proof} From \eqref{i4} it follows that ${\rm int} \sigma_{{\bf R}_i}(T)=\emptyset$ and so, according to
  (\ref{n0}), (\ref{n1}), (\ref{n2}), (\ref{glava-}), we get $\sigma_{{\bf gDR
R}_i}(T)=\sigma_{gKR}(T)$.
\end{proof}

From Corollary \ref{r=} it follows that if $\sigma(T)$ is most
countable or  contained in a line, then
$\sigma_{gKR}(T)=\sigma_{gDR}(T)$. Every  self-adjoint, as well as,
unitary operator on Hilbert space have the spectrum  contained in a
line.  The spectrum   of a polynomially Riesz operator \cite{ZDHB} or
polynomially meromorphic operator \cite{KaashoekPRIA} is most countable.

\bigskip

Obviously, $\sigma_{gKR}(T)\subset \sigma_{gK}(T)$ and the following corollary is an improvement of   \cite[Corollary 4.2, (iv), (v), (vi)]{ZC}.
\begin{corollary} \label{cor5}
For $T \in L(X)$ the following inclusions hold:
\begin{eqnarray*}
\inter \, \sigma_{ap}(T) \setminus \inter \, \sigma_{\mathcal{B}_+}(T) &\subset& \sigma_{gKR}(T), \\
  \inter \, \sigma_{su}(T) \setminus \inter \, \sigma_{\mathcal{B}_-}(T) &\subset& \sigma_{gKR}(T), \\
  \inter \, \sigma(T) \setminus \inter \, \sigma_{\mathcal{B}}(T) &\subset& \sigma_{gKR}(T).
  \end{eqnarray*}
 \end{corollary}
\begin{proof}
Follows from the equivalences  {\rm (ii)} $\Longleftrightarrow  $
{\rm (vi)} in
Theorems \ref{t1} and \ref{t2}, and the equivalence {\rm (ii)} $\Longleftrightarrow  $
{\rm (ix)} in Theorem \ref{t3}.
\end{proof}
\begin{corollary}\label{SVEP}  Let $T\in L(X)$ have  the SVEP. Then all accumulation points of $\sigma_{\B_+}(T)$ belong to $\sigma_{gKR}(T)$.
\end{corollary}
\begin{proof} Follows from the equivalence (iii)$\Longleftrightarrow$(v) of Theorem
\ref{t1}.
\end{proof}

\begin{corollary}\label{SVEP1}  Suppose that for $T\in L(X)$, $T^\prime$ has  the SVEP. Then all accumulation points of $\sigma_{\B_-}(T)$ belong to $\sigma_{gKR}(T)$.
\end{corollary}
\begin{proof} Follows from the equivalence (iii)$\Longleftrightarrow$(v) of Theorem
\ref{t2}.
\end{proof}
\begin{corollary}\label{SVEP2}  Suppose that  both $T$ and $T^\prime$ have  the SVEP. Then all accumulation points of $\sigma_{\B}(T)$ belong to $\sigma_{gKR}(T)$.
\end{corollary}
\begin{proof} Follows from the equivalence (iii)$\Longleftrightarrow$(viii) of Theorem
\ref{t3}.
\end{proof}
\begin{corollary} \label{example1}
Let $T$ be unilateral weighted right shift operator on $\ell_p(\NN)$, $1\le p<\infty$, with weight $(\omega_n)$, and let $c(T)=\lim\limits_{n\to\infty}\inf (\omega_1\cdot\dots\cdot \omega_n)^{1/n}=0$. Then $\sigma_{gKR}(T)=\sigma_{gDR{\bf R}_i}(T)= \sigma(T)=\overline{D(0,r(T))}$.
\end{corollary}
\begin{proof} According to \cite[Corollary 3.118]{Ai} it follows that
 $\sigma_{ap}(T)=\sigma_{\B_+}(T)
=\sigma(T)=\overline{D(0,r(T))}$  and $T$ 
has the SVEP. Since every $\lambda\in \sigma(T)$ is an accumulation point of $\sigma_{\B_+}(T)$, according to Corollary  \ref{SVEP} it follows that $\sigma_{gKR}(T)=\sigma_{{\bf gDR R}_i}(T)= \sigma(T)=\overline{D(0,r(T))}$.
\end{proof}

\begin{theorem}\label{partial0} Let $T\in L(X)$ and $4\le  i\le 12$. Then the following implication holds:
\begin{equation}\label{pos0}
  0\in\partial\sigma_{{\bf R}_i}(T)\ \wedge\ T\ {\rm admits\ a\ GKRD}\Longrightarrow 0\in{\iso}\, \sigma_{{\bf R}_i}(T).
\end{equation}
Moreover,
\begin{eqnarray}
  0\in\partial\sigma_{ap}(T)\ \wedge\ T\ {\rm admits\ a\ GKRD} &\Longrightarrow& 0\notin{\acc}\, \sigma_{\B_+}(T);\label{pos1} \\
   0\in\partial\sigma_{su}(T)\ \wedge\ T\ {\rm admits\ a\ GKRD} &\Longrightarrow& 0\notin{\acc}\, \sigma_{\B_-}(T);\label{pos2} \\
    0\in\partial\sigma(T)\ \wedge\ T\ {\rm admits\ a\ GKRD} &\Longrightarrow& 0\notin{\acc}\, \sigma_{\B}(T).\label{pos3}
\end{eqnarray}
 \end{theorem}
\begin{proof} Let
 $0\in\partial\sigma_{\B_+}(T)$ and let $T$ admit a GKRD. Then $0\in \sigma_{\B_+}(T)$ and  $0\notin \inter\, \sigma_{\B_+}(T)$. From the equivalence {\rm (v)}$\Longleftrightarrow${\rm (vi)} in Theorem \ref{t1}, it follows that $0\notin\acc\, \sigma_{\B_+}(T)$. Since $0\in \sigma_{\B_+}(T)$, it means that $0\in{\iso}\, \sigma_{\B_+}(T)$. The remaining cases ($5\le i\le 12$) can be proved in a similar way.

 Let
 $0\in\partial\sigma_{ap}(T)$ and let $T$ admit a GKRD. Then $0\in \sigma_{ap}(T)$ and  $0\notin \inter\, \sigma_{ap}(T)$. From the equivalence {\rm (ii)}$\Longleftrightarrow${\rm (v)} in Theorem \ref{t1}, it follows that $0\notin\acc\, \sigma_{\B_+}(T)$. Similarly for (\ref{pos2}) and (\ref{pos3}).
\end{proof}

\medskip

\begin{theorem}\label{partial} Let $T\in L(X)$. Then the following inclusions hold:
\begin{equation}\label{poss-0}
 \partial\sigma_{{\bf R}_i}(T)\cap\acc\, \sigma_{{\bf R}_i}(T)\subset \sigma_{gKR}(T) ,\ \ 4\le  i\le 12.
\end{equation}
Moreover,
\begin{eqnarray}
  \partial\sigma_{ap}(T)\cap {\acc}\, \sigma_{\B_+}(T)\subset\sigma_{gKR}(T);\label{ppos1} \\
   \partial\sigma_{su}(T)\cap {\acc}\, \sigma_{\B_-}(T)\subset\sigma_{gKR}(T);\label{ppos2} \\
    \partial\sigma(T)\cap {\acc}\, \sigma_{\B}(T)\subset\sigma_{gKR}(T).\label{ppos3}
\end{eqnarray}
\end{theorem}
\begin{proof}
Follows from Theorem \ref{partial0}.
\end{proof}
We remark that (\ref{poss-0}) is an improvement of \cite[Theorem 4.5 (4.3)]{ZC} for the case that  $  i \in\{4,\dots, 12\}$.

\medskip

It follows an example of a operator which does not admit a GKRD.
\begin{example} {\rm Let $X$ be an infinite dimensional Banach  space and let
$\tilde X=\oplus_{n=1}^{\infty} X_i$  where $X_i=X$, $i\in\NN$. Let $A=\oplus_{n=1}^{\infty}(\frac 1n I)$ where $I$ is identity on $X$. Then $A\in L(\tilde X)$ and    $\sigma (A)=\{0\}\cup\{\frac 1n:n\in\NN\}$. The operator $A-(1/n)$ has finite ascent and descent and hence $1/n$ is a pole of the resolvent of $A$, but $\alpha (A-(1/n))=\beta (A-(1/n))=+\infty$ and so,  $1/n$ is a pole of the infinite algebraic multiplicity and $1/n$ belongs to the Fredholm spectrum of $A$. Consequently, $A$ is meromorphic  and Fredholm spectrum of $A$ is equal to the spectrum of $A$. Therefore, $\sigma_{\B}(A)=\sigma_{\W}(A)=
\sigma_{\Phi}(A)=\sigma(A)$. From \eqref{ppos3} it follows that  $\{0\}=\partial\sigma(A)\cap {\acc}\, \sigma_{\B}(A)\subset\sigma_{gKR}(A)$. Since $\sigma_{gKR}(A)\subset\sigma_{gK}(A)\subset\sigma_{gD}(A)=\acc\, \sigma (A)=\{0\}$, we get $\sigma_{gKR}(A)=\{0\}$  and hence,  $A$ does not admit a GKRD.
Also, we remark that $\sigma_{\Phi_+}(A)=\sigma_{\Phi_-}(A)=\sigma_{\W_+}(A)=\sigma_{\W_-}(A)=\sigma_{\B_+}(A)=\sigma_{\B_-}(A)=\sigma_{ap}(A)=\sigma_{su}(A)=\sigma(A)$ and
 $0\notin{\rm int}\, \sigma(T)$. This means that for $T\in L(X)$ the condition that $0\notin {\rm int} \sigma_{{\bf R}_i}(T)$, $i\in\{1,\dots,12\}$ is not sufficient for $T$ to admit a GKRD. Therefore,   the condition that the operator admits a GKRD in the statements  (ii) and (ix) of Theorem \ref{t3}, as well as  in the statements (ii) and (vi) of Theorems \ref{t1} and \ref{t2}, and also, in the statements (iii) of Theorem \ref{t4}, can not be ommited.

}
\end{example}


\smallskip

The {\it connected hull}  of a compact subset $K$ of the complex
plane $\CC$, denoted by $\eta K$, is the complement of the unbounded
component of $\CC\setminus K$ \cite[Definition
7.10.1]{H}. Given a compact subset $K$ of the plane, a hole of $K$
is a bounded component of $\CC\setminus K$, and so a hole of $K$ is
a component of $\eta K\setminus K$.

We recall  that, for compact subsets $H,K\subset\CC$, the following implication holds (\cite[Theorem
7.10.3]{H}):
\begin{equation}\label{spec.1}
\partial H\subset K\subset
H\Longrightarrow\partial H\subset\partial K\subset K\subset
H\subset\eta K=\eta H\ .
\end{equation}
Evidently, if $K\subseteq\CC$ is finite, then $\eta K=K$.
Therefore, for compact subsets $H,K\subseteq\CC$, if
$\eta K=\eta H$, then $H\ {\rm is\
finite}$ if and only if $ K\ {\rm is\ finite}$,
and in that case $H=K$.

\medskip

\begin{theorem}\label{rub} Let $T\in L(X)$. Then

\begin{center}
\begin{tabular}{rcccl}
& & $\partial \sigma_{\bf gDR\mathcal{M}}(T) \; \; \subset \; \; \partial \sigma_{\bf gDR\mathcal{W}_+}(T) \; \; \subset \; \; \partial \sigma_{\bf gDR\Phi_+}(T)$& & \\
 & \rotatebox{30}{$\subset$} & &\rotatebox{-30}{$\subset$}& \\
$\partial \sigma_{gDR}(T)$
& $\subset$ &
$\partial \sigma_{\bf gDR\mathcal{W}}(T) \; \; \; \; \; \; \; \subset \; \; \; \; \; \; \; \partial \sigma_{\bf gDR\Phi}(T)$& $\subset$ &$\partial \sigma_{gKR}(T)$\\
 &\rotatebox{-30}{$\subset$} & & \rotatebox{30}{$\subset$} &\\
& & $\partial \sigma_{\bf gDR\mathcal{Q}}(T) \; \; \subset \; \; \partial \sigma_{\bf gDR\mathcal{W}_-}(T) \; \; \subset \; \; \partial \sigma_{\bf gDR\Phi_-}(T)$ & &\\

\end{tabular}
\end{center}

\begin{eqnarray*}
&\partial\sigma_{\bf gDR \Phi}(T)\subset \partial\sigma_{\bf gD R\Phi_+}(T),\ \  \partial\sigma_{\bf gD R\Phi}(T)\subset \partial\sigma_{\bf gD R\Phi_-}(T), \\
  &\partial\sigma_{\bf gDR \W}(T)\subset \partial\sigma_{\bf gDR \W_+}(T),\ \  \partial\sigma_{\bf gD R\W}(T)\subset \partial\sigma_{\bf gD R\W_-}(T),
  \end{eqnarray*}

and

\begin{eqnarray}%
\eta\sigma_{gKR}(T)&=&\eta\sigma_{\bf gDR\Phi_+}(T)=\eta\sigma_{\bf gDR\W_+}(T)=\eta\sigma_{\bf gDR\M}(T)\nonumber\\&=&\eta\sigma_{\bf gDR\Phi_-}(T)
=\eta\sigma_{\bf gDR\W_-}(T)=\eta\sigma_{\bf gDR\Q}(T)\label{rub2}
\\&=&\eta\sigma_{\bf gDR\Phi}(T)=\eta\sigma_{\bf gDR\W}(T)=\eta\sigma_{gDR}(T)
.\nonumber
\end{eqnarray}
\end{theorem}
\begin{proof} According to (\ref{spec.1}) and the inclusions (\ref{inkluzija-}) and (\ref{inkluzija--})
 it is sufficient to prove the inclusions
\begin{equation*}
 \partial\sigma_{{\bf gDRR_i}}(T)\subset \sigma_{gKR}(T),\ i\in\{1,2,3,7,8,9,10,11,12\}.
\end{equation*}

%
%
%
%
%
%
%
%
%
%
%
%
%
Suppose that $\lambda_0\in \partial\sigma_{{\bf gDRR_i}}(T)$. Since $\sigma_{{\bf gDRR_i}}(T)$ is closed, it follows that
  \begin{equation}\label{mlad}
    \lambda_0\in \sigma_{{\bf gDRR_i}}(T)=\sigma_{gKR}(T)\cup{\rm int}\, \sigma_{{\bf R_i}}(T).
  \end{equation}

 We shall prove that
 \begin{equation}\label{mlad1}
   \lambda_0\notin {\rm int}\, \sigma_{\bf R_i}(T).
 \end{equation}
  Suppose  that $\lambda_0\in {\rm int}\, \sigma_{\bf R_i}(T)$.
   Then there exists an $\epsilon>0$ such that $D(\lambda_0,\epsilon)\subset \sigma_{\bf R_i}(T)$. This implies that $D(\lambda_0,\epsilon)\subset {\rm int}\,\sigma_{\bf R_i}(T)$ and hence, $D(\lambda_0,\epsilon)\subset \sigma_{{\bf gDRR_i}}(T)$, which contradicts to the fact that $\lambda_0\in \partial\sigma_{{\bf gDRR_i}}(T)$.
Now from  (\ref{mlad}) and (\ref{mlad1}), it follows that
$\lambda_0\in\sigma_{gKR}(T)$.
\end{proof}

From  (\ref{rub2}) it follows that $\sigma_{gKR}(T)$ is finite if and only if $\sigma_{{\bf g DR R}_i}(T)$ is finite for arbitrary
$i\in\{1,\dots ,12\}$,
and in that case $\sigma_{gKR}(T)=\sigma_{{\bf g DR R}_i}(T)$ for all $i\in\{1,\dots ,12\}$. Also, from (\ref{rub2}) it follows that $\sigma_{gKR}(T)=\emptyset$ if and only if
 $\sigma_{{\bf g DR R}_i}(T)=\emptyset$ where $i$ is one of $1,\dots, 12$. Moreover, the following theorem holds:

\begin{theorem}\label{Lak}
Let $T\in L(X)$.  The following statements are equivalent:

\snoi {\rm (i)} $\sigma_{gKR}(T)=\emptyset$;

\snoi {\rm (ii)} $\sigma_{gDR}(T)=\emptyset$;

\snoi {\rm (iii)} $T$ is polynomially Riesz;

\snoi  {\rm (iv)} $\sigma_{\B}(T)$ is a finite set;

\snoi  {\rm (v)} $\theta (T)$ is a finite set,  where $\theta$ is one of $\sigma_{eK}$, $\sigma_{\Phi_+}$, $\sigma_{\Phi_-}$, $\sigma_{\Phi}$, $\sigma_{\W_+}$,  $\sigma_{\W_-}$,  $\sigma_{\W}$, $\sigma_{\B_+}$, $\sigma_{\B_-}$.
\end{theorem}
\begin{proof} The equivalence {\rm (i)}$\Longleftrightarrow${\rm (ii)}  follows from (\ref{rub2}).

The equivalences
{\rm (iii)}$\Longleftrightarrow${\rm (iv)}$\Longleftrightarrow${\rm (v)}  have been proved in \cite{ZDHB}.

{\rm (ii)} $\Longrightarrow$ {\rm (iv)}:  Suppose that $\sigma_{gDR}(T)=\emptyset$. From \eqref{poziv} it follows that $\acc\, \sigma_{\B}(T)=\emptyset$ and so, $\sigma_{\B}(T)$ is a finite set.

{\rm (iii)} $\Longrightarrow$ {\rm (ii)}: Let $T$ be polynomially Riesz and $\pi_T^{-1}(0)=\{\lambda_1,\dots,\lambda_n\}$. According to \cite[Theorem 2.2]{ZDHB} we have that $\sigma_{\B}(T)=\pi_T^{-1}(0)$ where $\pi_T$ is the minimal polynomial of $T$. It implies that  $T-\lambda$ is Browder and hence, generalized Drazin-Riesz invertible     for every $\lambda\notin \pi_T^{-1}(0)$.
\ According to \cite[Theorem 2.13]{ZDHB}, $X$ is  decomposed into the
direct sum $X=X_1\oplus\dots\oplus X_n$ where $X_i$ is closed
$T$-invariant subspace of $X$,  $T=T_1\oplus\dots\oplus T_n$
where $T_i$ is the reduction of $T$ on $X_i$ and $T_i-\lambda_i$
is Riesz, $i=1,\dots, n$. Since $T_i-\lambda_i$ is Riesz, it follows that $\sigma_{\B}(T_i-\lambda_i)\subset \{0\}$ and hence,  $\sigma_{\B}(T_i)\subset \{\lambda_i\}$, $i=1,\dots,n$. It implies that $T_i-\lambda_j$ is Browder for $i\ne j$, $i,j\in\{1,\dots,n\}$.

The following  argument shows that $ X_2\oplus\dots\oplus X_n$ is closed. Consider Banach space $X_1\times X_2\times\dots\times X_n$ provided with the canonical norm $\|(x_1,\dots, x_n)\|=\sum_{i=1}^n\|x_i\|$, $x_i\in X_i$, $i=1,\dots,n$. Then the map  $f:X_1\times\dots\times X_n\to X_1\oplus\dots\oplus X_n=X$    defined by $f((x_1,\dots,x_n)) =x_1+\dots +x_n$, $x_i\in X_i$, $i=1,\dots,n$, is a
  homeomorphism 
   and since $\{0\}\times X_2\times\dots\times X_n$ is closed in  $X_1\times X_2\times\dots\times X_n$, it follows that $f(\{0\}\times X_2\times\dots\times X_n)=\{0\}\oplus X_2\oplus\dots\oplus X_n=X_2\oplus\dots\oplus X_n$ is closed.

Consider  the  decomposition
$$T-\lambda_1=(T_1-\lambda_1)\oplus(T_2-\lambda_1)\oplus\dots\oplus (T_n-\lambda_1).
$$
Since $(X_1, X_2\oplus\dots\oplus X_n)\in Red(T)$, $(T-\lambda_1)_{X_1}=T_1-\lambda_1$ is Riesz, and since  $(T-\lambda_1)_{X_2\oplus\dots\oplus X_n}=(T_2-\lambda_1)\oplus\dots\oplus (T_n-\lambda_1)$
is Browder as a direct sum of Browder operators $T_2-\lambda_1,\dots , T_n-\lambda_1$
(Lema \ref{lema} (i)),  it follows that $T-\lambda_1$ is generalized Drazin-Riesz invertible. In that way   we can  prove that $T-\lambda_i$ is generalized Drazin-Riesz invertible for every $i\in\{1,\dots,n\}$. Consequently, $T-\lambda$ is  generalized Drazin-Riesz invertible     for every $\lambda\in\CC$ and hence, $\sigma_{gDR}(T)=\emptyset$.
\end{proof}

\medskip

The inclusion $\sigma_{gKR}(T)\subset \sigma_{gK}(T)$
 can be proper as it can be seen  on the example of a Riesz operator $T$ with infinite spectrum. Namely, according to Theorem \ref{Lak} it follows that $\sigma_{gKR}(T)=\emptyset$, while $\sigma_{gK}(T)=\{0\}$. Moreover, if $T$ is polynomially Riesz with infinite spectrum, then $\sigma_{gKR}(T)=\emptyset$, while from \cite[Corollary 4.11]{ZC} we have that $\sigma_{gK}(T)=\sigma_{gD}(T)=\acc\, \sigma(T)\ne\emptyset$.

\medskip

The generalized Drazin-Riesz resolvent set of $T\in L(X)$ is defined by
$\rho_{gKR}(T)=\CC\setminus\sigma_{gKR}(T)$.

\begin{corollary} \label{componenta} Let $T\in L(X)$ and let $\rho_{gKR}(T)$ has only one component. Then
$$
\sigma_{gKR}(T)=\sigma_{gDR}(T).
$$
\end{corollary}
\begin{proof} Since
$\rho_{gKR}(T)$ has only one component, it follows that
$\sigma_{gKR}(T)$ has no holes, and so  $\sigma_{gKR}(T)=\eta
\sigma_{gKR}(T)$. From (\ref{rub2}) it follows that
$\sigma_{gDR}(T)\supset \sigma_{gKR}(T)=\eta \sigma_{gKR}(T)= \eta
\sigma_{gDR}(T)\supset \sigma_{gDR}(T)$ and hence,
$\sigma_{gDR}(T)=\sigma_{gKR}(T)$.
\end{proof}

\begin{theorem}\label{nov} Let $T\in L(X)$ and $4\le i\le 12$.
If
\begin{equation}\label{bon}
 \partial\sigma_{{\bf R}_i}(T)\subset\acc\, \sigma_{{\bf R}_i}(T),
\end{equation}
then
%
\begin{eqnarray}
 \partial\sigma_{{\bf R}_i}(T)&\subset&\sigma_{gKR}(T)\subset \sigma_{gK}(T)\subset \sigma_{Kt}(T)\subset \sigma_{eK}(T)\subset \sigma_{{\bf R}_i}(T),\label{part1}\\
 \partial\sigma_{{\bf R}_i}(T)&\subset&\sigma_{gKR}(T)\subset \sigma_{gK}(T)\subset \sigma_{{\bf g DR R}_i}(T)\subset  \sigma_{{\bf R}_i}(T)\label{part2}
\end{eqnarray}
and
\begin{equation}\label{pacc1}
  \eta\sigma_{{\bf R}_i}(T)=\eta \sigma_{gKR}(T)=\eta \sigma_{gK}(T)=\eta \sigma_{Kt}(T)=\eta \sigma_{eK}(T)=\eta\sigma_{{\bf g DR R}_i}(T).
\end{equation}
If $1\le i\le 3$, then
\begin{equation}\label{pariz}
 \partial\sigma_{{\bf R}_i}(T)\subset{\acc}\, \sigma_{{\bf  R}_{i+3}}(T)
\end{equation}
  implies
\begin{equation}\label{pacc0*}
  \partial\sigma_{{\bf R}_i}(T)\subset \sigma_{gKR}(T)\subset \sigma_{gK}(T)\subset \sigma_{Kt}(T)\subset \sigma_{eK}(T)\subset\sigma_{K}(T)\subset \sigma_{{\bf R}_i}(T)
\end{equation}
and
\begin{equation*}
  \eta\sigma_{{\bf R}_i}(T)=\eta \sigma_{gKR}(T)=\eta \sigma_{gK}(T)=\eta \sigma_{Kt}(T)=\eta \sigma_{eK}(T)=\eta \sigma_{K}(T)=\eta\sigma_{{\bf g DR R}_i}(T).
\end{equation*}
\end{theorem}
\begin{proof}
From $\partial\sigma_{{\bf R}_i}(T)\subset\acc\, \sigma_{{\bf R}_i}(T)$ it follows that $\partial\sigma_{{\bf R}_i}(T)\cap\acc\, \sigma_{{\bf R}_i}(T)=\partial\sigma_{{\bf R}_i}(T)$, and so  from (\ref{poss-0}) it follows that $\partial\sigma_{{\bf R}_i}(T)\subset \sigma_{gKR}(T)$. 
(\ref{pacc1}) follows from (\ref{part1}),  (\ref{part2}) and \eqref{spec.1}.


For
$1\le i\le 3$, remark that $\sigma_{K}(T)\subset \sigma_{{\bf R}_i}(T)$
 and we can proceed analogously as above.
\end{proof}

The Goldberg spectrum of $T\in L(X)$ is defined by
$$\sigma_{ec}(T)=\{\lambda\in\CC:R(T-\lambda)\ {\rm is\ not\ closed}\}.
$$
Obviously, $\sigma_{ec}(T)\subset\sigma_{R_i}(T)$ for all
$i\in\{1,\dots,12\}$.
\begin{theorem}\label{Ben}  Let $T\in L(X)$ and $4\le i\le 12$. If
\begin{equation*}
\sigma_{{\bf R}_i}(T)=\partial\sigma_{{\bf R}_i}(T)=\acc\,
\sigma_{{\bf R}_i}(T),
\end{equation*}
 then
 \begin{equation*}
\sigma_{ec}(T)\subset\sigma_{gKR}(T)=
\sigma_{gK}(T)=\sigma_{Kt}(T)=\sigma_{eK}(T)=\sigma_{{\bf
R}_i}(T)=\sigma_{{\bf g DR R}_i}(T).
\end{equation*}
For $1\le i\le 3$,
\begin{equation*}
  \sigma_{{\bf R}_i}(T)=\partial\sigma_{{\bf R}_i}(T)\subset\acc\,
\sigma_{{\bf R}_{i+3}}(T),
\end{equation*}
 implies
\begin{equation*}
\sigma_{ec}(T)\subset\sigma_{gKR}(T)=
\sigma_{gK}(T)=\sigma_{Kt}(T)=\sigma_{eK}(T)=\sigma_{K}(T)=\sigma_{{\bf
R}_i}(T)=\sigma_{{\bf g DR R}_i}(T).
\end{equation*}
\end{theorem}
\begin{proof} Suppose that $\sigma_{ap}(T)=\partial\sigma_{ap}(T)$ and that every $\lambda\in  \sigma_{ap}(T)$ is an accumulation point of  $\sigma_{\B_+}(T)$.
From Theorem \ref{nov} it follows that
\begin{eqnarray*}
  \sigma_{ap}(T)=\partial\sigma_{ap}(T)\subset\sigma_{gKR}(T)\subset \sigma_{gK}(T)\subset \sigma_{Kt}(T)\subset \sigma_{eK}(T)\subset\sigma_{K}(T)\subset \sigma_{ap}(T),\\
  \sigma_{ap}(T)=\partial\sigma_{ap}(T)\subset \sigma_{gKR}(T)\subset \sigma_{\bf g DR \M}(T)\subset \sigma_{ap}(T),
\end{eqnarray*}
and so
$
\sigma_{ec}(T)\subset\sigma_{ap}(T)=\sigma_{gKR}(T)=\sigma_{gK}(T)=\sigma_{Kt}(T)=\sigma_{eK}(T)=\sigma_{K}(T)=
\sigma_{\bf g DR \M}(T).
$

The other cases ($i=2,\dots,12$) can be proved similarly.
\end{proof}

We remark that if $K\subset\CC$ is compact, then
for $\lambda\in\partial K$  there is equivalence:
\begin{equation}\label{zz}
 \lambda\in \acc\, K\Longleftrightarrow \lambda\in \acc\, \partial K.
\end{equation}

\begin{theorem} \label{PO} Let $T\in L(X)$ be an operator for which $\sigma_{ap}(T)=\partial\sigma (T)$ and every $\lambda\in \partial\sigma (T)$ is not isolated in $\sigma (T)$. Then
$\sigma_{ec}(T)\subset\sigma_{ap}(T)=\sigma_{gKR}(T)=\sigma_{gK}(T)=\sigma_{Kt}(T)=\sigma_{eK}(T)=\sigma_{K}(T)=\sigma_{\bf g DR \M}(T)
$.
\end{theorem}
\begin{proof}
From $\sigma_{ap}(T)=\partial\sigma (T)$, since $\partial\sigma(T)\subset\partial\sigma_{ap}(T)\subset\sigma_{ap}(T)$   it follows
that $\sigma_{ap}(T)=\partial\sigma_{ap}(T)$, while from (\ref{zz})
it follows that every $\lambda\in \partial\sigma (T)$ is not
isolated in $\partial\sigma (T)$. Therefore, every $\lambda\in \partial\sigma (T)$ is not
isolated  in $\sigma_{ap}(T)$ and hence, $\sigma_{ap}(T)\subset\acc\, \sigma_{ap}(T)$. Since $\acc\, \sigma_{ap}(T)\subset\sigma_{\B_+}(T)\subset \sigma_{ap}(T)$ \cite[Corollary 20.20]{Mu}, we get that $\sigma_{ap}(T)=\acc\, \sigma_{ap}(T)=\sigma_{\B_+}(T)$. Hence $\acc\,  \sigma_{\B_+}(T)=\acc\, \sigma_{ap}(T)=\sigma_{ap}(T)=\partial\sigma_{ap}(T)$. Now from
Theorem \ref{Ben}  we get $\sigma_{ec}(T)\subset\sigma_{ap}(T)=\sigma_{gKR}(T)=\sigma_{gK}(T)=\sigma_{Kt}(T)=\sigma_{eK}(T)=\sigma_{K}(T)=\sigma_{g DR \M}(T)$.
\end{proof}
\begin{theorem}\label{kl} Let $T\in L(X)$ be an operator for which $\sigma_{su}(T)=\partial\sigma (T)$ and every $\lambda\in \partial\sigma (T)$ is not isolated in $\sigma (T)$. Then
$\sigma_{ec}(T)\subset\sigma_{su}(T)=\sigma_{gKR}(T)=\sigma_{gK}(T)=\sigma_{Kt}(T)=\sigma_{eK}(T)=\sigma_{K}(T)=\sigma_{\bf g DR \Q}(T).$
\end{theorem}
\begin{proof}
Follows from the inclusions $\partial\sigma(T)\subset\partial\sigma_{su}(T)\subset\sigma_{su}(T)$, (\ref{zz})  and  Theorem \ref{Ben},
analogously to the proof of Theorem \ref{PO}.
\end{proof}

\begin{example}
{\em If $X$ is one of $c_{0}(\ZZ)$ and $\ell_{p}(\ZZ)$, $p\ge 1$,
then for the forward and backward bilateral shifts $W_1,\ W_2\in
L(X)$ there are equalities
\begin{eqnarray}
 \sigma_{ap}(W_1)=\sigma_{su}(W_1)=\sigma (W_1)=\partial\DD,\label{za1}
 \\\sigma_{ap}(W_2)=\sigma_{su}(W_2)=\sigma (W_2)=\partial\DD.\label{za2}
\end{eqnarray}
For every $i=1,\dots ,12$, from Theorem \ref{Ben} (or Theorems \ref{PO} and \ref{kl}), \eqref{za1},
\eqref{za2} and \eqref{inkluzija} it follows  that
\begin{eqnarray*}
 &\sigma_{gKR}(W_1)=
 \sigma_{{\bf g DR R}_i}(W_1)=\sigma_{g DR }(W_1)=
 \sigma (W_1)=\partial\DD,
   \\
  &\sigma_{gKR}(W_2)=
 \sigma_{{\bf g D R}_i}(W_2)=\sigma_{g DR }(W_2)=\sigma (W_2)=\partial\DD.
\end{eqnarray*}
It  follows also from Corollary \ref{r=} and the inclusions in 
\eqref{inkluzija-}. }
\end{example}
\begin{example}{\em For  each  $X\in \{  c_0(\NN), c(\NN), \ell_{\infty}(\NN),
\ell_p(\NN)\}$, $p\ge 1$, and the forward and backward unilateral shifts
$U$, $V\in L(X)$ there are equalities $\sigma(U)=\sigma(V)=\DD$,  $
\sigma_{ap}(U)=\sigma_{su}(V)=\partial\DD,$ and hence, $
\sigma_{ap}(U)=\partial \sigma(U)\subset \acc\, \sigma(U)$ and
$\sigma_{su}(V)=\partial\sigma(V)\subset \acc\, \sigma(V)$. From
Theorems  \ref{PO} and \ref{kl}, and \eqref{inkluzija} we get
$  \sigma_{gKR}(U)=  \sigma_{gK}(U)=\sigma_{K}(U)=\sigma_{ap}(U)=\sigma_{\bf g DR\M}(U)=\sigma_{\bf g DR\W_+}(U)=\sigma_{\bf g DR\Phi_+}(U)=\partial\DD$ and
   $\sigma_{gKR}(V)=  \sigma_{gK}(V)=\sigma_{K}(V)=\sigma_{su}(V)=\sigma_{\bf g DR\Q}(V)=\sigma_{\bf g DR\W_-}(V)=\sigma_{\bf g DR\Phi_-}(V)=\partial\DD$.

Remark that the forward  unilateral shift
$U$ is non-invertible isometry.
In \cite{aienarosas}, p. 187,  it is noticed that every  non-invertible isometry $T$ has the property that $\sigma(T)=\overline{D(0,r(T))}$ and $\sigma_{ap}(T)=\partial D(0,r(T))$, and hence   $\sigma_{ap}(T)=\partial\sigma (T)$ and every $\lambda\in \partial\sigma (T)$ is not isolated in $\sigma (T)$. Therefore, according to Theorem \ref{PO}, for arbitrary non-invertible isometry $T$ we get that
$  \sigma_{gKR}(T)=\sigma_{\bf g DR\Phi_+}(T)=\sigma_{\bf g DR\W_+}(T)=\sigma_{\bf g DR \M}(T)=\sigma_{ap}(T)= \partial D(0,r(T))$.
}
\end{example}

\begin{example}{\em For the $Ces\acute{a}ro\ operator$ $C_p$ defined on the classical Hardy space $H_p(\DDD)$, $\DDD$ the open unit disc and $1<p<\infty$, by
$$
 (C_pf)(\lambda)=\ds\frac 1{\lambda}\int_0^{\lambda}\ds\frac{f(\mu)}{1-\mu}\, d\mu,\ \, {\rm for\ all\ }f\in H_p(\DDD)\ {\rm and\ }\lambda\in\DDD,
 $$
it is known that its spectrum is  the closed disc $\Gamma_p$ centered at $p/2$ with radius $p/2$ and $\sigma_{Kt}(C_p)=\sigma_{ap}(C_p)=\partial \Gamma_p$ \cite{Mill}, \cite{aienarosas}.
According to Theorem \ref{PO} we get that $\sigma_{gKR}(C_p)=\sigma_{gK}(C_p)=\sigma_{ap}(C_p)=\partial \Gamma_p$.

}
\end{example}

\medskip \noindent
\author{Sne\v zana
\v C. \v Zivkovi\'c-Zlatanovi\'c}

\noindent{University of Ni\v s\\
Faculty of Sciences and Mathematics\\
P.O. Box 224, 18000 Ni\v s, Serbia}

\noindent {\it E-mail}: {\tt mladvlad@mts.rs}

\bigskip
\noindent
\author{Milo\v s D. Cvetkovi\'c}

\noindent University of Ni\v s\\
Faculty of Sciences and Mathematics\\
P.O. Box 224, 18000 Ni\v s, Serbia

\noindent {\it E-mail}: {\tt miloscvetkovic83@gmail.com}

\end{document}